\documentclass[12pt]{article}

\usepackage{amsmath,amsfonts,amssymb,amsthm}
 \usepackage{amscd,stmaryrd,mathabx,tikz-cd}
 \usepackage[utf8]{inputenc}
 \usepackage{xcolor}
  \usepackage[margin=1.9cm]{geometry}

\title{Keisler Measures and Generically Stable Random Types}
 \date{\today}
 \author{Karim Khanaki\thanks{Partially supported by IPM grant 1405030024}\\
 Arak University of Technology, and\\
 Institute for Research in Fundamental Sciences (IPM)}

\newtheorem{Theorem}{Theorem}[section]
 \newtheorem{Proposition}[Theorem]{Proposition}
 \newtheorem{Definition}[Theorem]{Definition}
 \newtheorem{Remark}[Theorem]{Remark}
 \newtheorem{Lemma}[Theorem]{Lemma}
 \newtheorem{Corollary}[Theorem]{Corollary}
 \newtheorem{Fact}[Theorem]{Fact}
 \newtheorem{Example}[Theorem]{Example}
 \newtheorem{Question}[Theorem]{Question}
 
 \newtheorem{Observation}[Theorem]{Observation}

\begin{document}
 \maketitle

\begin{abstract}
 {\normalsize
 We introduce the notions of $rgs$ and $irgs$ for Keisler
 measures, motivated by the study of generically stable random
 types and their associated Morley sequences. We obtain
 characterizations of these notions in terms of averages of
 classical first-order formulas over suitable probabilistic
 partitions (Theorems~\ref{main-rgs} and~\ref{main-irgs}).
 We compare these notions with $fim$, $fam$, and self-averaging,
 and show that for types the notions $fim$, $irgs$, and $rgs$
 coincide.
 We prove that if a definable global Keisler measure
 $\mu$ is not dependent, then every Morley sequence
 in its canonical random extension $r_\mu$ is divergent
 (Theorem~\ref{nondependent-divergent}).
 Consequently, every $irgs$ measure is dependent
 (Corollary~\ref{irgs-dependent}) and hence symmetric
 (Corollary~\ref{cor}).
 Furthermore, we show that for $irgs$ measures
 the model-theoretic instability events
 $\mathbf{O}^\varphi$, $\mathbf{I}^\varphi$,
 and $\mathbf{L}^\varphi$ have
 $\mathbb{P}_\mu$-measure zero
 (Theorem~\ref{irgs-tame-event}).
 In view of the equivalence of $fim$ and $irgs$
 subsequently established in \cite{CGH2026},
 this provides an alternative proof of the
 corresponding result for $fim$ measures in \cite{GH}.
 }
 \end{abstract}

\section{Introduction}

Generic stability has a long history in model theory,
with its origins in classical stability theory and
its subsequent development in the study of NIP
theories. The underlying ideas are closely related
to definability, invariance, indiscernibility, and
convergence, and have played an important role in
understanding the behavior of types in various
model-theoretic settings. Pillay and Tanovic
\cite{PT} introduced the notion of generically
stable types in arbitrary theories, extending
a concept that had previously been studied
primarily in stable and NIP contexts.

On the other hand, the study of Keisler measures
has developed considerably in recent years,
motivated in part by connections with probability,
functional analysis, and learning theory.
In NIP theories, generic stability of measures
admits several equivalent characterizations,
including simultaneous definability and finite satisfiability, 
and suitable approximation properties.
In arbitrary theories, however, the situation
is more delicate. Notions such as definability,
finite satisfiability, finite approximation
($fam$), and frequency interpretation ($fim$)
need not coincide; see, for example,
\cite{CGH-wild}. Thus, identifying a suitable
analogue of generic stability for measures
in arbitrary theories has been a natural
problem in the subject.

An additional difficulty arises from the role
of Morley sequences. For global types, the
behavior of Morley sequences is fundamental
to the theory of generic stability. For
Keisler measures, although Morley products
provide an important analogue, there is no
corresponding sequence of realizations in
the ordinary model-theoretic sense.
This suggests studying measures through
randomization, where a Keisler measure can
be associated with a type in a continuous
theory. Such an approach makes it possible
to investigate questions about measures
using the theory of generic stability
and convergence of types.

The randomization of first-order structures
was developed by Keisler and by Ben Yaacov
and Keisler \cite{Kei,BK}. In particular,
Ben Yaacov introduced a canonical construction
associating a random type $r_\mu$ to a
definable Keisler measure $\mu$
\cite{B-transfer}. This construction and
its relationship with Morley products
and generic stability were subsequently
studied in several works, including
\cite{K-Morley,K-generic,CGH,G-transfer}.
Conant, Gannon, and Hanson \cite{CGH}
investigated the notion of self-averaging
and established important implications
between $fim$ measures, generic stability
of their canonical random extensions,
and self-averaging. These developments
provided a natural setting for further
investigation of generic stability
through random types.

Motivated by this approach, and by the
study of convergence in continuous logic
\cite{K-generic} and dependent measures
\cite{K-dependent}, we introduce in this
paper two notions of generic stability
for Keisler measures, called
\emph{randomly generically stable}
($rgs$) and
\emph{independently randomly generically
stable} ($irgs$).
The distinction between them reflects
the difference between arbitrary random
extensions and the canonical random
extension. More precisely, a definable
global measure $\mu$ is $rgs$ if it admits
a generically stable random extension,
whereas it is $irgs$ if its canonical
random extension $r_\mu$ is generically
stable. Clearly,
\[
\mathrm{irgs}\Longrightarrow\mathrm{rgs},
\]
while the converse remains an interesting
open question.

Our first main results characterize these
notions in terms of sequences of atomic
measures associated with measurable
partitions of an atomless probability
space. We introduce stable basic sequences
and independent stable basic sequences,
which encode suitable uniform convergence
properties of classical first-order
formulas. We show that a measure is $rgs$
if and only if it admits a stable basic
sequence (Theorem~\ref{main-rgs}), and
that it is $irgs$ if and only if it
admits an independent stable basic
sequence (Theorem~\ref{main-irgs}).
These characterizations provide a
connection between approximation of
Keisler measures and generic stability
of types in the randomization.
Moreover, for global types, the notions
$rgs$ and $irgs$ coincide with the
usual notion of generic stability.

In the initial version of this paper,
which appeared on arXiv in May 2026
(\texttt{arXiv:2605.15870}), these notions
and their basic-sequence characterizations
were introduced and compared with $fim$,
$fam$, and self-averaging. In particular,
the implications
\[
\mathrm{fim}
\Longrightarrow
\mathrm{irgs}
\Longrightarrow
\mathrm{rgs}
\Longrightarrow
\mathrm{fam}
\]
were considered, and it was shown that
every $irgs$ measure is dependent in the
sense of \cite{K-dependent}. Subsequently,
Conant, Gannon, and Hanson \cite{CGH2026}
established that, for Borel-definable
global Keisler measures, the notions
of $fim$, generic stability of the
canonical random extension, and
self-averaging are equivalent. In our
terminology, their theorem yields
\[
\mathrm{fim}
\Longleftrightarrow
\mathrm{irgs}
\Longleftrightarrow
\mathrm{self\text{-}averaging}.
\]
Thus, some of the implications obtained
in the earlier investigations are now
known to be equivalences. The distinction
between $rgs$ and $irgs$, however, is
not resolved by this result. Indeed,
the notion of $rgs$ and its relationship
with $irgs$ are also discussed in
\cite[Section~6.2]{CGH2026}.
We retain the original framework and
approach, which offer a complementary
perspective on generic stability
through random extensions and their
approximation properties.

Another main theme of this paper is
the relationship between dependence
of Keisler measures and convergence
of Morley sequences in the randomization.
The notion of a dependent measure,
introduced in \cite{K-dependent},
is based on the measure-theoretic
stability theory of Fremlin and
Talagrand. It provides a useful
connection between model-theoretic
properties of measures and classical
questions concerning families of
measurable functions.

In the present treatment, we establish
a divergence theorem which strengthens
the convergence-theoretic formulation
of the dependence result. More precisely,
we prove that if a definable global
Keisler measure $\mu$ is not dependent,
then every Morley sequence in its
canonical random extension $r_\mu$
is divergent
(Theorem~\ref{nondependent-divergent}).
The proof combines the measure-theoretic
characterization of non-dependence
with the product-distribution properties
of Morley sequences in the canonical
randomization. This result provides
a direct connection between
non-dependence of a classical measure
and failure of convergence in its
associated random type.

As a consequence, every $irgs$
measure is dependent
(Corollary~\ref{irgs-dependent}),
and hence satisfies the corresponding
symmetry properties of Morley products
(Corollary~\ref{cor}). In view of
the subsequent equivalence of $fim$
and $irgs$ established in
\cite{CGH2026}, the dependence
consequence is also known for $fim$
measures. The divergence theorem,
however, gives a more explicit
description of the underlying
convergence phenomenon and provides
a different proof of the dependence
consequence.

We also investigate the behavior
of $irgs$ measures with respect to
the model-theoretic instability
events introduced by Gannon and
Hanson \cite{GH}. These events
correspond to the order property,
the independence property, and
the strict order property along
Morley products of a measure.
We prove that for every $irgs$
measure $\mu$ and every formula
$\varphi(x;y)$,
\[
\mathbb P_\mu(\mathbf O^\varphi)
=
\mathbb P_\mu(\mathbf I^\varphi)
=
\mathbb P_\mu(\mathbf L^\varphi)
=0
\]
(Theorem~\ref{irgs-tame-event}).
Our proof is based on generic
stability of the canonical random
extension and uniform convergence
of its Morley sequences. Since
$fim$ and $irgs$ are now known
to coincide, this provides an
alternative approach to the
corresponding result of
\cite{GH}, rather than a
generalization to a strictly
larger class of measures.

Finally, we study the relationship
between the order and independence
events under symmetry of Morley
powers, and discuss examples
illustrating the limitations of
dependence in controlling these
events. These results further
highlight the role of randomization
in relating generic stability,
dependence, and convergence
phenomena for Keisler measures.

The paper is organized as follows.
Section~2 recalls the necessary
background on generically stable
types in continuous logic,
randomizations, and canonical random
extensions, and introduces stable
basic sequences and independent
stable basic sequences. Section~3
defines $rgs$ and $irgs$, establishes
their characterizations, and compares
them with $fim$, $fam$, and
self-averaging. Section~4 proves
the divergence theorem and derives
its consequences for dependence
and symmetry. Section~5 investigates
model-theoretic instability events,
their relationship under symmetry,
and several illustrative examples.
The appendix contains additional
technical details supporting the
proof of the divergence theorem.

\section{Preliminaries}

In this section, we review the basic material needed for
 the proofs of the main theorems. We first review the notion
 of generically stable types in continuous logic (and hence
 in discrete logic), together with the equivalent
 characterizations that will be used throughout the paper.
 We then introduce simple models of randomization and random
 types. And at the end of this section, we introduce notions
 of sequences with properties that capture what we think a
 Morley sequence for a measure ought to have.

We let $T$ be a classical first-order theory in a countable
 language $L$ and let $M$ be a model of $T$.
 To simplify notation, we assume throughout that all
 partitioned formulas are of the form $\varphi(x;y)$, where
 $|x|=1$ and $y$ is a finite tuple of variables.
 Since we also work with continuous theories, it is convenient
 to interpret $L$-formulas as $\{0,1\}$-valued functions:
 for $a,b\in M$, we set $\varphi(a,b)=1$ if
 $M\models\varphi(a,b)$, and $\varphi(a,b)=0$ otherwise.

We denote the monster model of the classical theory by
 $\mathcal{U}$, and, to avoid confusion, we always denote
 the monster model of the continuous theory by
 $\mathcal{C}$. In this paper, the continuous theory under
 consideration is essentially $T^R$, the randomization of
 the theory $T$.

Whenever no confusion arises, we do not distinguish
 notationally between discrete and continuous formulas.
 Nevertheless, formulas in the randomization theory are
 typically denoted by
 $\llbracket\varphi(\ldots)\rrbracket$, where $\varphi$
 is a classical formula. These formulas will be defined later.

We denote elements of simple models of randomization by
 $f,g,\ldots$. Such elements are expressions of the form
 $\sum_i a_i^{A_i}$, which will be defined precisely in due
 course. Elements of a classical model $M$, such as $a$,
 are identified with the corresponding constant random
 variables and denoted by $a^\Omega$. Whenever no confusion
 arises, we simply write $a$ instead of $a^\Omega$.

\subsection{Generically stable types in continuous logic}

The definition of generically stable types in continuous
 logic was first given in \cite{K-generic}, although it is
 essentially a straightforward generalization of such types
 in discrete logic. These types were also introduced and
 studied in \cite[Def.~3.1]{CGH}. In the definition below,
 if we replace \(r<s\) simply by \(0,1\), we essentially
 obtain the definition of this concept in discrete logic.
 We recall again that in this paper the monster model of
 continuous theories will be denoted by \(\mathcal{C}\)
 and that of discrete theories by \(\mathcal{U}\), in order
 to avoid confusion.

\begin{Definition}
 Suppose $p\in S_x(\mathcal{C})$ is invariant over
 $M\preceq\mathcal{C}$. We say that $p$ is
 \emph{generically stable (over $M$)} if there do not exist
 a formula $\varphi(x;y)$, a Morley sequence
 $(a_i)_{i<\omega}$ in $p$ over $M$, a sequence
 $(b_i)_{i<\omega}$ in $\mathcal{C}^y$, and real numbers
 $r<s$, such that
 \[
 \varphi(a_i;b_j)\le r\quad\text{if }i\le j,
 \]
 and
 \[
 \varphi(a_i;b_j)\ge s\quad\text{if }i>j.
 \]
 \end{Definition}

The concepts of definable types, finitely satisfiable types,
 \emph{fam} and \emph{fim} can be defined analogously to
 discrete logic in classical logic. For example, for the
 definition of \emph{fim} in continuous logic, refer to
 \cite{K-generic,CGH}. Analogous to the discrete case,
 $fim$ and generic stability are equivalent, and this,
 together with other equivalent conditions, is stated in
 the following fact.

\begin{Fact}[\cite{K-generic,CGH}]
 \label{generic-fact-1}
 Let $M\preceq\mathcal{C}$ and suppose
 $p\in S_x(\mathcal{C})$ is $M$-invariant.
 Then the following are equivalent.
 \begin{enumerate}
 \item[(i)] $p$ is generically stable over $M$.
 \item[(ii)] For any Morley sequence $(a_i)_{i<\omega}$
 in $p$ over $M$, any $L$-formula $\varphi(x;y)$,
 and any $b\in\mathcal{C}^y$,
 \[
 \lim_{i\to\infty}\varphi(a_i;b)=\varphi^p(x;b).
 \]
 \item[(iii)] For any $L$-formula $\varphi(x;y)$ and any
 $\varepsilon>0$, there exists
 $n_{\varphi,\varepsilon}<\omega$ such that for any
 Morley sequence $(a_i)_{i<\omega}$ in $p$ over $M$
 and any $b\in\mathcal{C}^y$,
 \[
 \left|\left\{i<\omega:
 \left|\varphi(a_i;b)-\varphi^p(x;b)\right|
 \ge\varepsilon\right\}\right|
 \le n_{\varphi,\varepsilon}. \qquad (*)
 \]
 \item[(iv)] $p$ is $fim$ over $M$.
 \end{enumerate}
 \end{Fact}

In this paper, we call the convergence given in $(*)$
 in part (iii) of the above fact \emph{uniform convergence}.

\begin{Remark}
 We observe that \emph{uniform convergence}, as introduced
 above, is stronger than the Baire-$1/2$ convergence
 introduced in \cite{K-generic}. The latter guarantees only
 the convergence of Morley sequences, whereas uniform
 convergence further implies the definability of $p$.
 \footnote{We emphasize that the notion of uniform
 convergence employed here is entirely different from
 the classical notion of uniform convergence in analysis;
 no confusion should arise between these two concepts.
 The terminology is adopted solely by analogy and should
 not be interpreted as identifying them.}
 \end{Remark}

In what follows, further equivalent conditions for generic
 stability are presented, the most important of which---playing
 a key role in this paper---is condition (ii). This condition
 essentially provides a sequence inside a small model that
 encodes all the information of Morley sequences and gives us
 the control we need. In fact, instead of speaking of Morley
 sequences, we speak of {\em the} Morley sequence inside
 the model.

\begin{Fact}\label{key-fact}
 \cite[Thm.~2.13, Cor.~2.14]{K-generic}
 Let $T$ be a continuous theory, $M$ a small model of $T$,
 and let $p\in S_x(\mathcal{C})$ be a global $M$-invariant
 type. The following are equivalent:
 \begin{enumerate}
 \item[(i)] $p$ is generically stable over $M$.
 \item[(ii)] There exists a sequence $(c_i)_{i<\omega}$
 in $M$ which converges \emph{uniformly} to $p$ in
 the following sense: for any formula $\varphi(x;y)$
 and any $\varepsilon>0$, there exists
 $n_{\varphi,\varepsilon}<\omega$ such that for every
 $b\in\mathcal{C}^y$,
 \[
 \left|\left\{i<\omega:
 \left|\varphi(c_i;b)-\varphi^p(x;b)\right|
 \ge\varepsilon\right\}\right|
 \le n_{\varphi,\varepsilon}.
 \]
 \item[(iii)] $p$ is definable over and finitely
 satisfiable in some small model, and there exists
 a convergent Morley sequence in $p$ over $M$.
 \end{enumerate}
 \end{Fact}

\begin{Remark}
 In part (ii) of the above fact, when working in classical
 logic, the eventual Ehrenfeucht--Mostowski type (EEM-type)
 of $(c_i)$ over $M$ is, as shown in \cite{Kyle-seq},
 the Morley type $p^{(\omega)}|_M$. In \cite{K-generic},
 the notion of the EEM-type for symmetric formulas
 (the SEEM-type) was introduced in continuous logic.
 We note that since the type $p$ is definable and finitely
 satisfiable, its Morley sequence is totally indiscernible;
 hence, in this setting we use the EEM-type rather than
 the SEEM-type.
 \end{Remark}

\subsection{Randomization}

In this subsection, we review some notions concerning
 the randomization of a theory and its simple randomization
 models that will be used throughout this paper.
 For a more detailed exposition of randomization in
 continuous logic, we refer the reader to the original
 reference \cite{BK}.

Let \(T\) be a complete first-order theory and
 \(M\models T\). Fix an atomless probability space
 \((\Omega,\mathcal B,\mathcal{P})\).
 The simple randomization model associated to \(M\),
 denoted by \(M_0^\Omega\), consists of all measurable
 functions \(f:\Omega\to M\) having finite range.
 Every element \(f\in M_0^\Omega\) may therefore be
 written in the form
 \[
 f=\sum_{i=1}^n a_i^{A_i},
 \]
 where \(a_1,\dots,a_n\in M\) and
 \((A_1,\dots,A_n)\) is a measurable partition of
 \(\Omega\).

For every classical formula
 \(\varphi(x_1,\dots,x_n)\) and every
 \(f_1,\dots,f_n\in M_0^\Omega\), one defines the event
 \[
 \llbracket\varphi(f_1,\dots,f_n)\rrbracket
 =
 \{\omega\in\Omega:
 M\models\varphi(f_1(\omega),\dots,f_n(\omega))\}.
 \]

The probability of this event gives rise to a definable
 predicate
 \[
 \mathbb P\llbracket\varphi(f_1,\dots,f_n)\rrbracket
 =
 \mathcal P(\llbracket
 \varphi(f_1,\dots,f_n)\rrbracket).
 \]

Hence every classical formula induces a real-valued
 predicate on the randomization structure.
 The metric on \(M_0^\Omega\) is defined by
 \[
 d(f,g)=\mathcal P(\llbracket f\neq g\rrbracket).
 \]
 The completion of \(M_0^\Omega\) with respect to this
 metric is denoted by \(M^\Omega\) and is called the
 randomization of \(M\).

An important point is that the interpretation of formulas
 extends continuously from simple random variables to the
 completion. Consequently, every formula of the original
 language determines a uniformly continuous predicate on
 \(M^\Omega\).

If \(a\in M\), we identify \(a\) with the constant
 random variable \(a^\Omega:\omega\mapsto a\).
 Thus \(M\) naturally embeds into \(M^\Omega\).
 Whenever no confusion may arise, we write simply $a$
 instead of $a^\Omega$ in randomization formulas.

It is known that the theory of all simple randomization
 structures associated with models of \(T\) is complete.
 We denote this theory by \(T^R\).
 The randomization construction and the theory \(T^R\)
 were introduced by \cite{BK}. Moreover, the theory
 \(T^R\) admits quantifier elimination (see Theorem~2.9
 there). 
In particular, every formula of $T^R$ whose free
variables are in the random-element sort can be
uniformly approximated by finite continuous
combinations of basic predicates of the form
\[
\mathbb P\llbracket\varphi(\bar x)\rrbracket,
\]
where $\varphi(\bar x)$ is a classical $L$-formula.
This observation will allow us to pass from
uniform convergence for basic predicates to
uniform convergence for arbitrary formulas
of the randomization.
 We denote the monster model of $T^R$ by $\mathcal C$.
 Moreover, if $\mathcal U$ is a monster model of $T$
 and $M\models T$, then
 \[
 M^\Omega\preceq\mathcal U^\Omega
 \preceq\mathcal C.
 \]
 Notice that, in general (except for very tame theories
 $T$), the structure $\mathcal U^\Omega$ is not saturated;
 see, for example, Example~3.10(ii) and Theorem~3.12 in
 \cite{BK}.

For further emphasis and clarity, we restate the
 quantifier elimination of $T^R$ in the form of a fact.

\begin{Fact}[\cite{BK}, Thm.~2.9]
The theory $T^R$ admits quantifier elimination.
In particular, the type of a random element
$f\in\mathcal C$ over a randomization model
$A$ is determined by the values
\[
\mathbb P\llbracket\varphi(f;\bar g)\rrbracket,
\]
where $\varphi$ ranges over classical
$L$-formulas and $\bar g$ ranges over
tuples of random elements of $A$.
\end{Fact}

In what follows, we recall a method for the canonical
 extension of a Keisler measure in $T$ to a random type
 in $T^R$, which plays a key role in this paper and was
 introduced by Ben Yaacov in \cite{B-transfer}.
 This canonical extension has also been studied in
 \cite{CGH} and \cite{G-transfer}.

Consider a global measure $\mu$ that is definable over
 the model $M$ of $T$. We define the random type $r_\mu$,
 which we call the canonical extension of $\mu$
 (Ben Yaacov called it the natural extension), as follows.

The definition is carried out in two stages: first for
 parameters in $M_0^\Omega$, and then, using the definability
 of $\mu$, for all parameters in $\mathcal C$.
 Consider the formula
 $\mathbb P\llbracket\varphi(x;y)\rrbracket$
 and the element $g=\sum_{i=1}^k a_i^{A_i}$ of
 $M_0^\Omega$. We define
 \[
 \mathbb P^{r_\mu}\llbracket\varphi(x;g)\rrbracket
 =
 \sum_{i=1}^k\mathbb P(A_i)
 \mu(\varphi(x;a_i)).
 \]
 Now for an arbitrary parameter $g\in\mathcal C$,
 we take a net of parameters $(g_j)$ in $M_0^\Omega$
 such that
 $\operatorname{tp}(g_j/M_0^\Omega)\to
 \operatorname{tp}(g/M_0^\Omega)$.
 We define
 \[
 \mathbb P^{r_\mu}\llbracket\varphi(x;g)\rrbracket
 =
 \lim_j\mathbb P^{r_\mu}
 \llbracket\varphi(x;g_j)\rrbracket.
 \]

The following fact appears in \cite{B-transfer}, part
 of which was originally presented in Ben Yaacov's note
 essentially without proof, while the part concerning
 the Morley product is new in \cite{CGH}.

\begin{Fact}\label{r_mu-fact}
 Let $\mu$ be a global measure definable over $M$.
 \begin{enumerate}
 \item[(i)] \cite{B-transfer,CGH} $r_\mu$ is well-defined
 and definable over $M_0^\Omega$.
 \item[(ii)] \cite[Corollary~3.16]{CGH} For every $n$,
 $(r_\mu)^{(n)}=r_{\mu^{(n)}}$.
 \item[(iii)] \cite[Corollary~3.19]{CGH} If $\mu$ is
 $fim$, then $r_\mu$ is also $fim$.
 \item[(iv)] \cite[Proposition~3.20]{CGH} If, moreover,
 $\mu$ is a type, then $\mu$ is $fim$ if and only if
 $r_\mu$ is $fim$.
 \end{enumerate}
 \end{Fact}

\subsection{Morley Sequences for Measures}

In this section, we introduce a notion which, based
 on the previous observations, may serve as a candidate
 for a Morley sequence for measures.
 In the next section, we investigate the connection of
 this notion with generically stable random types.

\begin{Definition}\label{atomic sequence}
 {\em
 Let
 $\bar r=((r_{i,j})_{j=1}^{\infty}:i<\omega)$
 be an array of real numbers in $(0,1]$ such that
 for all $i$,
 $\sum_{j=1}^\infty r_{i,j}=1$.
 A sequence $(\mu_i)_{i=1}^\infty$ of global measures
 is called a sequence of $M$-atomic measures with
 weights $\bar r$ if for every formula $\varphi(x;y)$
 and every $b\in\mathcal U$,
 \[
 \mu_i(\varphi(x;b))
 =
 \sum_{j=1}^{\infty}
 r_{i,j}\varphi(a_{i,j};b),
 \]
 where
 $(\bar a_i=(a_{i,j})_{j<\omega}:i<\omega)$
 are elements of $M$, and for each $i$ only finitely
 many of the $a_{i,j}$ are distinct.
 \footnote{This condition that for each $i$ only finitely
 many $a_{i,j}$ are distinct is not essential and can
 be removed; nevertheless, we have included it so that
 each $\mu_i$ is close to an average measure, which is
 a well-known concept in model theory.}
 In this case we say that $\mu_i$ is built upon
 $\bar r$ and $\bar a_i$, and write
 $\mu_i=\sum_j r_{i,j}\cdot a_{i,j}$
 (for all $i$).
 }
 \end{Definition}

For a sequence $(\mu_i(x))_{i=1}^\infty$ of global
 measures, we say that $(\mu_i)_{i=1}^\infty$
 converges pointwise to a global measure $\mu$ and write
 $\lim_{i\to\infty}\mu_i=\mu$ if for every formula
 $\varphi(x)$ with parameters in $\mathcal U$ we have
 $\lim_{i\to\infty}\mu_i(\varphi(x))=\mu(\varphi(x))$.

In this section,
 $(\Omega,\mathcal B,\mathcal P)$ is an arbitrary fixed
 atomless probability space.
 \footnote{Although we could present the definition in a
 way that makes no reference to a probability space, namely
 only with coefficients of real numbers satisfying certain
 equations, there is a risk that, firstly, these coefficients
 may not be consistent, and otherwise our definition would
 become stronger than being equivalent to the corresponding
 random type being generically stable.}
 By a (finite or countable) partition of $\Omega$, we
 mean a partition into measurable subsets (events) of
 positive measure. We shall omit the term ``measurable''
 throughout whenever no ambiguity arises.

The following definition presents a concept that seems a
 suitable candidate for a Morley sequence of a Keisler
 measure; we will later show that it corresponds to the
 Morley sequence of its extended random type.

\begin{Definition}[Stable basic sequence]\label{sbs}
 {\em
 Let $\mu(x)$ be a global Keisler measure and $M$
 a small model of $T$. We say that $\mu$ has a
 \emph{stable basic sequence over $M$} if there exists
 a sequence $(\mu_i(x))_{i=1}^\infty$ of $M$-atomic
 measures with weights $\bar r$ (of the form given in
 Definition~\ref{atomic sequence}, i.e.,
 $\mu_i=\sum_j r_{i,j}\cdot a_{i,j}$)
 such that for all $i$,
 $(r_{i,j})_{j=1}^\infty$ corresponds to a partition
 $(A_{i,j})_{j=1}^\infty$ of $\Omega$ into measurable
 sets such that $\mathcal P(A_{i,j})=r_{i,j}$
 and conditions (i) and (ii) below hold:
 \begin{itemize}
 \item[(i)] \textbf{Extension.}
 $\lim_{i\to\infty}\mu_i=\mu$, and
 \item[(ii)] \textbf{Stability.}
 For every formula $\varphi(x;y)$ and every
 $\varepsilon>0$, there exists a natural number
 $n_{\varphi,\varepsilon}$ such that for every
 $k\in\mathbb N$, every measurable partition
 $\bar k=(B_1,\ldots,B_k)$ of $\Omega$
 and every $\bar b=(b_1,\dots,b_k)\in M^k$,
 there exists a real number
 $m_\varepsilon:=m_{\varphi,\varepsilon,\bar k,\bar b}$
 such that
 \[
 \Bigl|
 \bigl\{i:
 \bigl|\sum_{s=1}^{k}\sum_{j=1}^{\infty}
 r_{i,j}^s\varphi(a_{i,j};b_s)
 -m_\varepsilon\bigr|>\varepsilon
 \bigr\}
 \Bigr|
 <n_{\varphi,\varepsilon},
 \]
 where $r_{i,j}^s=\mathcal P(A_{i,j}\cap B_s)$.
 \end{itemize}
 }
 \end{Definition}

\begin{Remark}
 In Theorem~\ref{main-rgs}, we will show that every
 measure which admits a stable basic sequence has an
 extension to a generically stable random type $p_\mu$.
 In fact, the information of Morley sequences of this
 random type is encoded by the stable basic sequence
 introduced above. Therefore, it is not unreasonable
 to refer to a stable basic sequence as a Morley sequence
 of $p_\mu$ (or, equivalently, a Morley sequence
 associated to $\mu$).
 \end{Remark}

We first record the following elementary
consequence of the existence of a stable
basic sequence. In the next section,
we prove the stronger characterization
of such measures in terms of $rgs$,
which also yields the $fam$ property
(Corollary~\ref{comparing}(i)).
\begin{Fact}
 Assume that $\mu$ admits a stable basic sequence over $M$.
 Then:
 \begin{enumerate}
 \item[(i)] $\mu$ is definable over $M$.
 \item[(ii)] $\mu$ is finitely satisfiable in $M$.
 \end{enumerate}
 \end{Fact}

\begin{proof}
 Assume that $(\mu_i)$ is a stable basic sequence over
 $M$ for $\mu$, as in Definition~\ref{sbs}.

(i): Fix a formula $\varphi(x;y)$ and
$\varepsilon>0$, and put
$N=n_{\varphi,\varepsilon}$ as in
Definition~\ref{sbs}(ii), taking $k=1$.
Write
\[
u_i(b)=\mu_i(\varphi(x;b)).
\]
Each $u_i$ is definable over $M$, since
$\mu_i$ is a finite convex combination
of Dirac measures concentrated in $M$.

The stability condition extends from
$b\in M$ to arbitrary $b\in\mathcal U$:
for each finite set of indices, the
existence of a center in $[0,1]$ with
at most $N$ exceptional values is a
closed condition on $\operatorname{tp}(b/M)$.
Density of the types realized in $M$,
followed by compactness of $[0,1]$,
therefore gives such a center for the
entire sequence.

Since $u_i(b)\to\mu(\varphi(x;b))$,
this center differs from the limit
by at most $\varepsilon$. Consequently,
\[
H_\varepsilon(b)=
\operatorname{median}
(u_1(b),\ldots,u_{2N+1}(b))
\]
satisfies
\[
|H_\varepsilon(b)-\mu(\varphi(x;b))|
\leq 2\varepsilon
\]
uniformly in $b\in\mathcal U$.
(Here, $H_\varepsilon(b)$ is  the
$(N+1)$-st smallest element of
\(
u_1(b),\ldots,u_{2N+1}(b).
\).)
As $H_\varepsilon$ is definable over
$M$, the measure $\mu$ is definable
over $M$.

(ii): Since each $\mu_i$ can be approximated by average
 measures of elements from $M$, and since $\mu$ is the
 limit of the sequence $(\mu_i)$, it follows that
 $\mu$ is finitely satisfiable in $M$.
 \end{proof}

By Proposition~3.4 of \cite{Kyle-seq} together with
 the above observation, it follows that $\mu$ is $fam$
 over $M$, although, as mentioned earlier, a substantially
 stronger result will be established in the next section.

The next definition imposes an additional condition
 on a stable basic sequence, ensuring that the sequence
 eventually becomes independent of the choice of events
 in the underlying probability space. This point will
 become clearer later.

\begin{Definition}[Independent stable basic sequence]
 \label{isbs}
 {\em
 Let $\mu(x)$ be a global Keisler measure and $M$ a
 small model of $T$. We say that $\mu$ has an
 \emph{independent stable basic sequence over $M$}
 if there exists a sequence $(\mu_i(x))_{i=1}^{\infty}$
 of $M$-atomic measures with weights $\bar r$ (of the
 form given in Definition~\ref{atomic sequence},
 i.e., $\mu_i=\sum_j r_{i,j}\cdot a_{i,j}$)
 such that it satisfies conditions (i)--(ii) of
 Definition~\ref{sbs} and, in addition, fulfills
 the following independence condition:
 \begin{itemize}
 \item[(iii)] \textbf{Probabilistic independence.}
 For every formula $\varphi(x;y)$, every
 $\varepsilon>0$, every $k\in\mathbb N$, every
 $k$-partition $\bar k=(B_1,\ldots,B_k)$ of
 $\Omega$ into measurable sets, and every
 $\bar b=(b_1,\dots,b_k)\in M^k$,
 and let
 $m_\varepsilon:=m_{\varphi,\varepsilon,\bar k,\bar b}$
 be as in condition (ii), then we have
 \[
 m=\sum_{i=1}^k r_i\mu(\varphi(x;b_i)),
 \]
where
\[
m=\lim_{\varepsilon\to0}
m_{\varphi,\varepsilon,\bar k,\bar b},
\qquad
r_i=\mathcal P(B_i).
\]
The limit exists and is independent of the
choice of the centers. Indeed, condition (ii)
implies that the corresponding sequence of
averages converges to some limit $L$, and
every admissible center
$m_{\varphi,\varepsilon,\bar k,\bar b}$
differs from $L$ by at most $\varepsilon$.
 \end{itemize}
 }
 \end{Definition}

\begin{Remark}
 In Theorem~\ref{main-irgs}, we will show that any measure
 admitting an independent stable basic sequence has the
 property that its canonical extension $r_\mu$, previously
 defined in this paper, is generically stable.
 \end{Remark}

\section{$rgs$ and $irgs$ measures}

In this section, we introduce the main notions of the
 present paper, namely $rgs$ and $irgs$, and then provide
 characterizations of these concepts using the preliminaries
 developed earlier. We subsequently compare these notions
 with several well-known concepts, including $fim$, $fam$,
 and related notions.

\medskip
 Let $\mu(x)$ be a global Keisler measure in a classical
 theory $T$. A global random type $p_\mu$ is called a
 \emph{random extension} of $\mu$ if for every $L$-formula
 $\varphi(x;y)$ and every parameter
 $b\in\mathcal U^{|y|}$, we have
 \[
 \mathbb P^{p_\mu}
 \llbracket\varphi(x;b^\Omega)\rrbracket
 =\mu(\varphi(x;b)).
 \]

\medskip
 We now present the main definition of this paper.

\begin{Definition}\label{rgs-irgs}
 Let $\mu(x)$ be a global measure definable over a
 small model $M$.
 \begin{enumerate}
 \item[(i)] We say that $\mu(x)$ is
 \emph{randomly generically stable ($rgs$) over $M$}
 if there exists a random type extension $p_\mu$ of
 $\mu$ such that $p_\mu$ is generically stable over
 the randomization model $M^\Omega$ (for some atomless
 probability space $(\Omega,\mathcal P)$).
 \item[(ii)] We say that $\mu(x)$ is
 \emph{independently randomly generically stable
 ($irgs$) over $M$} if the canonical random type
 extension $r_\mu$ of $\mu$, defined earlier, is
 generically stable over $M^\Omega$ (for some atomless
 probability space $(\Omega,\mathcal P)$).
 \end{enumerate}
 \end{Definition}

The following theorem provides a characterization of
 $rgs$ measures in terms of the preliminaries developed
 earlier.

\begin{Theorem}\label{main-rgs}
 Let $\mu(x)$ be a global measure and $M$ a small model.
 Then the following conditions are equivalent:
 \begin{enumerate}
 \item[(i)] $\mu$ is $rgs$ over $M$.
 \item[(ii)] $\mu$ admits a stable basic sequence over $M$.
 \end{enumerate}
 \end{Theorem}

\begin{proof}
 (ii) $\Rightarrow$ (i):
 Suppose $\mu$ admits the stable basic sequence
 $(\mu_i=\sum_{j=1}^\infty r_{i,j}a_{i,j}:i<\omega)$
 over $M$ (with respect to the atomless probability
 space $(\Omega,\mathcal P)$). That is, $(\mu_i:i<\omega)$
 satisfies the conditions of Definition~\ref{sbs}.
 Consider and fix the simple randomization model
 $M_0^\Omega$.

By definition, there exists a sequence of countable
 partitions
 $((A_{i,j})_{j=1}^\infty)_{i<\omega}$
 of $\Omega$ such that
 $\mathbb P(A_{i,j})=r_{i,j}$ for each $i,j$.
 Define the sequence $(f_i)_{i=1}^\infty$ in
 $M_0^\Omega$ by
 $f_i=\sum_{j=1}^\infty a_{i,j}^{A_{i,j}}$
 for each $i$. We will show that this sequence defines
 an extension random type $p_\mu$ which is generically
 stable over $M^\Omega$.

First, for every formula $\varphi(x;y)$ and every
 $b\in\mathcal U$ we have
 \begin{align*}
 \lim_{i\to\infty}
 \mathbb P\llbracket\varphi(f_i;b^\Omega)\rrbracket
 &=\lim_{i\to\infty}\sum_{j=1}^\infty
 \mathbb P(A_{i,j})\varphi(a_{i,j};b)\\
 &=\lim_{i\to\infty}\mu_i(\varphi(x;b))\\
 &=\mu(\varphi(x;b)).\qquad(\dagger)
 \end{align*}

Secondly, assume that a positive $\varepsilon$ is given
 and $g\in M_0^\Omega$ is of the form
 $\sum_{s=1}^k b_s^{B_s}$ where $(B_s)_{s=1}^k$ is a
 partition of measurable sets in $\Omega$ and
 $\bar b=(b_s)_{s=1}^k\in M$.
 By assumption, there exist a natural number
 $n_{\varphi,\varepsilon}$ (which depends only on
 $\varphi$ and $\varepsilon$, and not on $g$)
 and a real number
 $m=m_{\varphi,\varepsilon,\bar k,\bar b}$
 with the property that
 \[
 \Bigl|\bigl\{i:
 \bigl|\sum_{s=1}^k\sum_{j=1}^\infty
 r_{i,j}^s\varphi(a_{i,j};b_s)-m\bigr|
 >\varepsilon\bigr\}\Bigr|
 <n_{\varphi,\varepsilon},\qquad(*)
 \]
 where $r_{i,j}^s=\mathcal P(A_{i,j}\cap B_s)$.

On the other hand, a simple calculation shows that
 \begin{align*}
 \mathbb P\llbracket\varphi(f_i;g)\rrbracket
 &=\sum_{s=1}^k
 \mathbb P(\llbracket\varphi(f_i;g)\rrbracket
 \cap B_s)\\
 &=\sum_{s=1}^k
 \mathbb P(\llbracket\varphi(f_i;b_s)\rrbracket
 \cap B_s)\\
 &=\sum_{s=1}^k\sum_{j=1}^\infty
 \mathbb P(A_{i,j}\cap B_s)\varphi(a_{i,j};b_s)\\
 &=\sum_{s=1}^k\sum_{j=1}^\infty
 r_{i,j}^s\varphi(a_{i,j};b_s).
 \end{align*}

We now extend the stability condition $(*)$
to arbitrary parameters $g\in\mathcal C$.
Fix $\varphi$ and $\delta>0$, and put
$N=n_{\varphi,\delta}$. Write
\[
u_i(g)=\mathbb P
\llbracket\varphi(f_i;g)\rrbracket.
\]
For a finite set $F\subseteq\omega$,
consider the condition
\[
\exists t\in[0,1]\quad
\left|\left\{i\in F:
|u_i(g)-t|>\delta\right\}\right|<N.
\]

Since all $u_i(g)$ belong to $[0,1]$,
projecting any admissible real center
onto $[0,1]$ does not increase its
distance from any $u_i(g)$.
This defines a closed subset of
$S_y(M^\Omega)$: it is a finite
disjunction of conditions obtained by
projecting closed subsets of
$S_y(M^\Omega)\times[0,1]$.
By $(*)$, it holds for all simple random
parameters. Since the types of such
parameters are dense in $S_y(M^\Omega)$,
it holds for every $g\in\mathcal C$.

Now fix $g\in\mathcal C$. The sets of
centers $t\in[0,1]$ satisfying the
above condition for the various finite
sets $F$ are closed and have the finite
intersection property. Compactness of
$[0,1]$ therefore gives a single center
$t_{g,\delta}$ such that
\[
\left|\left\{i<\omega:
|u_i(g)-t_{g,\delta}|>\delta
\right\}\right|<N.
\]
Since this holds for every $\delta>0$,
the sequence $(u_i(g))_{i<\omega}$
converges to a limit $L(g)$.

Moreover, applying the argument with
$\delta=\varepsilon/3$ gives
\[
\left|\left\{i<\omega:
|u_i(g)-L(g)|\geq\varepsilon
\right\}\right|
<n_{\varphi,\varepsilon/3}.
\]
By quantifier elimination for $T^R$, the same
uniform exceptional-index bound extends to
arbitrary formulas of the randomization.
Indeed, every such formula is uniformly
approximable by a continuous combination
of finitely many basic predicates. Given
$\varepsilon>0$, uniform continuity of this
combination provides a $\delta>0$ such that
$\delta$-closeness for all the basic predicates
implies $\varepsilon$-closeness for the
combination. The number of exceptional
indices is therefore bounded by the sum
of the corresponding finitely many bounds,
independently of the parameter.

Parameters from the event sort cause no
additional difficulty. Indeed, if $M$
contains two distinct elements $a_0,a_1$,
every event $B$ can be coded by a random
element $h$ satisfying
\[
B=\llbracket h=a_1^\Omega\rrbracket.
\]
Thus formulas involving event parameters
can be expressed using random-element
parameters and constants from $M^\Omega$.
Since the exceptional-index bounds are
uniform in all parameters, they remain
valid for these formulas as well.
If $M$ has only one element, the
random-element sort is trivial and the
assertion follows immediately.

Moreover, the median argument used above
shows that the limiting predicates are
definable over $M^\Omega$: each is uniformly
approximated by the median of finitely many
predicates with parameters from $M^\Omega$.
Consequently, the resulting limit type is
definable, and hence invariant, over
$M^\Omega$.

This shows that the sequence of types
 $\operatorname{tp}(f_i/\mathcal C)$ is convergent,
 and we denote its limit by $p_\mu$. Therefore,
 $p_\mu$ is a well-defined type which is moreover
 finitely satisfiable in $M^\Omega$.

In conclusion, for any formula $\varphi(x;y)$ and any
 $\varepsilon>0$, there exists
 $n_{\varphi,\varepsilon}<\omega$ such that for every
 $g\in\mathcal C^y$,
 \[
 \left|\left\{i<\omega:
 \left|\mathbb P\llbracket\varphi(f_i,g)\rrbracket
 -\mathbb P^{p_\mu}
 \llbracket\varphi(x;g)\rrbracket\right|
 \ge\varepsilon\right\}\right|
 \le n_{\varphi,\varepsilon}.\qquad(**)
 \]

By $(\dagger)$ and $(**)$, $p_\mu$ extends $\mu$.
 Also, the sequence $(f_i)$ is uniformly convergent
 according to $(**)$. By the equivalence of (i) and (ii)
 in Fact~\ref{key-fact}, it follows that $p_\mu$ is
 generically stable over $M^\Omega$. Therefore,
 $\mu$ is $rgs$ over $M$, and the proof is complete.

\medskip

(i) $\Rightarrow$ (ii):
Suppose that $\mu$ is $rgs$ over $M$,
and let $p_\mu$ be a generically stable
random extension of $\mu$ over $M^\Omega$.
By Fact~\ref{key-fact}, there exists a
sequence $(c_i)_{i<\omega}$ in $M^\Omega$
converging uniformly to $p_\mu$.

By the density of $M_0^\Omega$ in
$M^\Omega$, choose $f_i\in M_0^\Omega$
such that $d(f_i,c_i)<2^{-i}$.
Uniform continuity of formulas shows
that $(f_i)$ still converges uniformly
to $p_\mu$, possibly with a larger
bound on the exceptional indices.

Writing
\[
f_i=\sum_j a_{i,j}^{A_{i,j}},
\]
define
\[
\mu_i=\sum_j\mathcal P(A_{i,j})
\delta_{a_{i,j}}.
\]
Since $p_\mu$ extends $\mu$, the
sequence $(\mu_i)$ converges pointwise
to $\mu$. Moreover, the uniform
convergence of $(f_i)$ gives precisely
the stability condition in
Definition~\ref{sbs}, with the limiting
value of each randomization formula
as the corresponding center.
Thus $(\mu_i)$ is a stable basic
sequence for $\mu$ over $M$.
 \end{proof}

The following theorem provides a characterization of
 $irgs$ measures analogous to that of $rgs$.

\begin{Theorem}\label{main-irgs}
 Let $\mu(x)$ be a global measure and $M$ a small model.
 Then the following conditions are equivalent:
 \begin{enumerate}
 \item[(i)] $\mu$ is $irgs$ over $M$.
 \item[(ii)] $\mu$ admits an independent stable basic
 sequence over $M$.
 \end{enumerate}
 \end{Theorem}

\begin{proof}
 (ii) $\Rightarrow$ (i):
 In this case, $\mu$ has a stable basic sequence with
 the probabilistic independence property.
 Let $p_\mu$ be the generically stable random type
 defined in the proof of Theorem~\ref{main-rgs},
 associated to the independent stable basic sequence
 of $\mu$. For every
 $g=\sum_{j=1}^k b_j^{B_j}$ in $M_0^\Omega$,
 we have
 \begin{align*}
 \mathbb P^{r_\mu}\llbracket\varphi(x;g)\rrbracket
 &=\mathbb P^{r_\mu}
 \llbracket\varphi(x;\sum_{j=1}^k b_j^{B_j})\rrbracket\\
 &=\sum_{j=1}^k\mathbb P(B_j)
 \mu(\varphi(x;b_j))\\
 &=\sum_{j=1}^k\mathbb P(B_j)
 \mathbb P^{p_\mu}\llbracket\varphi(x;b_j)\rrbracket\\
 &\overset{(1)}{=}
 \mathbb P^{p_\mu}\llbracket\varphi(x;g)\rrbracket.
 \qquad(*)
 \end{align*}
 Equality (1) follows from the probabilistic independence
 property in Definition~\ref{isbs}.

Now, since both $p_\mu$ and $r_\mu$ are definable over
 $M^\Omega$ (and $M_0^\Omega$ is dense in this model),
 the equality $(*)$ holds for every $g$ in the monster
 model $\mathcal C$ as well; that is, $r_\mu=p_\mu$.
 Hence $r_\mu$ is generically stable, and $\mu$ is
 $irgs$ over $M$.

(i) $\Rightarrow$ (ii):
Suppose that $\mu$ is $irgs$ over $M$.
Then its canonical random extension
$r_\mu$ is generically stable over
$M^\Omega$.

Applying the construction in the
proof of Theorem~\ref{main-rgs}
to $p_\mu=r_\mu$, we obtain a stable
basic sequence $(\mu_i)$ for $\mu$.

Moreover, for every simple random
parameter
\[
g=\sum_{j=1}^k b_j^{B_j},
\]
the corresponding sequence converges
to
\[
\begin{aligned}
\mathbb P^{r_\mu}
\llbracket\varphi(x;g)\rrbracket
&=\sum_{j=1}^k\mathcal P(B_j)
\mu(\varphi(x;b_j)).
\end{aligned}
\]
Thus the centers in the stability
condition can be chosen to equal
this limit. Consequently, condition
(iii) of Definition~\ref{isbs}
holds, and $(\mu_i)$ is an
independent stable basic sequence.
 \end{proof}

For a global Keisler measure $\mu$ definable over
$M$, the canonical random extension $r_\mu$ is
the unique random extension of $\mu$ which is
definable over
\[
M'=\{a^\Omega:a\in M\}.
\]

Indeed, for each classical formula $\varphi(x;y)$,
let $d_\mu\varphi(y)$ be the $M$-definable
predicate representing
$b\mapsto\mu(\varphi(x;b))$.
This predicate is uniformly approximable by
finite linear combinations of characteristic
functions of $M$-definable sets. For every
simple random parameter $g$, we have
\[
\mathbb P^{r_\mu}
\llbracket\varphi(x;g)\rrbracket
=
\int_\Omega
d_\mu\varphi(g(\omega))\,d\mathcal P(\omega).
\]
Integrating the finite approximations gives
uniform approximations by finite linear
combinations of predicates of the form
\[
\mathbb P\llbracket
\psi(g;c^\Omega)\rrbracket,
\qquad c\in M.
\]
Thus the defining predicates of $r_\mu$
are definable over $M'$, first on simple
random parameters and then, by continuity,
on all parameters. Quantifier elimination
now yields the definability of $r_\mu$
over $M'$.

For uniqueness, let $p$ be any random extension
of $\mu$ definable over $M'$. In particular,
$p$ is $M'$-invariant. By the homogeneity of
the atomless probability algebra over constant
random variables, for every $b\in M$ and
every measurable event $B$, the value
\[
\mathbb P^p\bigl(
\llbracket\varphi(x;b^\Omega)\rrbracket
\cap B\bigr)
\]
depends only on $\mathcal P(B)$.
Finite additivity therefore implies that
\[
\mathbb P^p\bigl(
\llbracket\varphi(x;b^\Omega)\rrbracket
\cap B\bigr)
=
\mathcal P(B)\mu(\varphi(x;b)).
\]
Consequently, $p$ satisfies the canonical
product formula on simple random parameters,
and hence agrees with $r_\mu$ on $M_0^\Omega$.
Since both types are definable over
$M^\Omega$, they agree on all parameters.
Thus $p=r_\mu$.

In general, however, there is no reason to expect
 that every extension random type of $\mu$ which is
 definable over $M^\Omega$ must be $M'$-invariant.
 Nevertheless, this does hold in the special case
 where $\mu$ is a type, as shown below.

\begin{Fact}\label{p_mu=r_mu}
 Suppose $\mu$ is a global type definable over $M$.
 Then $r_\mu$ is the only random type extending
 $\mu$ that is definable over $M^\Omega$.
 \end{Fact}

\begin{proof}
 Let $p_\mu$ be a random type extending $\mu$ which
 is definable over $M^\Omega$.
 Since $\mu$ is a type, the equality
 \[
 \mathbb P^{p_\mu}\llbracket\varphi(x;g)\rrbracket
 =
 \mathbb P^{r_\mu}\llbracket\varphi(x;g)\rrbracket
 \]
 holds for every formula $\varphi(x;y)$ and every
 $g\in M_0^\Omega$.
 To verify this, suppose $g$ is of the form
 $\sum_{i=1}^k b_i^{B_i}$, where $B_1,\dots,B_k$
 is a partition of the measurable subsets of
 $\Omega$ and the $b_i$ are elements of $M$.
 Then
 \begin{align*}
 \mathbb P^{p_\mu}\llbracket\varphi(x;g)\rrbracket
 &=\mathbb P^{p_\mu}
 \Bigl\llbracket\varphi
 \bigl(x;\sum_{i=1}^k b_i^{B_i}\bigr)
 \Bigr\rrbracket\\
 &=\sum_{i=1}^k
 \mathbb P^{p_\mu}
 (\llbracket\varphi(x;b_i^\Omega)\rrbracket
 \cap B_i)\\
 &\stackrel{*}{=}
 \sum_{i=1}^k\mu(\varphi(x;b_i))
 \mathbb P(B_i)\\
 &\stackrel{\dagger}{=}
 \mathbb P^{r_\mu}
 \Bigl\llbracket\varphi
 \bigl(x;\sum_{i=1}^k b_i^{B_i}\bigr)
 \Bigr\rrbracket\\
 &=\mathbb P^{r_\mu}
 \llbracket\varphi(x;g)\rrbracket.
 \end{align*}
 The starred equality follows from the fact that
 $\mu$ is a type, and therefore
 $\mu(\varphi(x;b_i))$ is either $0$ or $1$.
 Recall that
 \[
 \mathbb P^{p_\mu}
 \llbracket\varphi(x;b_i^\Omega)\rrbracket
 =\mu(\varphi(x;b_i)).
 \]
 The equality marked by $\dagger$ is simply the
 definition of $r_\mu$.

Now, since $M_0^\Omega$ is dense in $M^\Omega$
 and both $p_\mu$ and $r_\mu$ are definable
 over $M^\Omega$, it follows that
 $p_\mu=r_\mu$.
 \end{proof}

The next proposition shows that, for types,
 the notions of $rgs$ and $irgs$ coincide,
 and gives another characterization of generic stability.

\begin{Proposition}\label{type-case}
 Suppose $\mu$ is a global type. Then the following
 are equivalent:
 \begin{enumerate}
 \item[(i)] $\mu$ is $rgs$ over $M$.
 \item[(ii)] $\mu$ is $irgs$ over $M$.
 \item[(iii)] $\mu$ is generically stable over $M$.
 \end{enumerate}
 \end{Proposition}

\begin{proof}
 Conditions (ii) and (iii) are equivalent by
 Fact~\ref{r_mu-fact}, together with the fact that
 $fim$ and generic stability coincide for types
 (Fact~\ref{generic-fact-1}).
 (ii) implies (i) by definition. The implication
 (i) $\Rightarrow$ (ii) follows from
 Fact~\ref{p_mu=r_mu}.
 \end{proof}

Before comparing these notions, we recall the concept
 of self-averaging, introduced in Definition~2.2 of
 \cite{CGH}. A global measure $\mu(x)$ definable
 over a model $M$ is said to be
 \emph{self-averaging (over $M$)} if, for every measure
 $\lambda=\lambda(x_i:i<\omega)$ satisfying
 $\lambda|_M=\mu^{(\omega)}|_M$, the measure $\lambda$
 converges to $\mu$ in the sense that, for every formula
 $\varphi(x)$ with parameters from the monster model
 $\mathcal U$,
 \[
 \lim_{i\to\infty}\lambda(\varphi(x_i))
 =\mu(\varphi(x)).
 \]

The following fact, essentially an adaptation of
 Remark~3.21 in \cite{CGH}, clarifies the relationship
 between the notions of $rgs$ and self-averaging.

\begin{Fact}[{\cite[Remark~3.21]{CGH}}]
 \label{self-fact}
 Let $\mu$ be a global Keisler measure
 definable over $M$, and put $p=r_\mu$.
 Then the following are equivalent:
 \begin{enumerate}
 \item[(i)]
 $\mu$ is self-averaging over $M$.

\item[(ii)]
 For every $L$-formula $\varphi(x;y)$
 and every $\varepsilon>0$, there is
 $n_{\varphi,\varepsilon}<\omega$ such that,
 for every Morley sequence
 $(f_i)_{i<\omega}$ in $p$ over $M^\Omega$
 and every $b\in\mathcal U^y$,
 \[
 \left|
 \left\{
 i<\omega:
 \left|
 \mathbb P\llbracket
 \varphi(f_i;b^\Omega)
 \rrbracket
 -
 \mu(\varphi(x;b))
 \right|
 >\varepsilon
 \right\}
 \right|
 \leq n_{\varphi,\varepsilon}.
 \]
 \end{enumerate}

In particular, every $irgs$ measure
 is self-averaging.
 \end{Fact}

\begin{proof}
 The equivalence of (i) and (ii)
 is established in
 \cite[Remark~3.21]{CGH}.

Now suppose that $\mu$ is $irgs$.
 By definition, $p=r_\mu$ is
 generically stable over $M^\Omega$.

Fix an $L$-formula $\varphi(x;y)$
 and $\varepsilon>0$, and consider
 the formula of the randomization
 \[
 \psi(x;y)=
 \mathbb P\llbracket
 \varphi(x;y)\rrbracket.
 \]

By Fact~\ref{generic-fact-1}(iii),
 there exists $n_{\varphi,\varepsilon}$
 such that for every Morley sequence
 $(f_i)_{i<\omega}$ in $p$ over
 $M^\Omega$ and every
 $g\in\mathcal C^y$,
 \[
 \left|
 \left\{
 i<\omega:
 |\psi(f_i;g)-\psi^p(x;g)|
 \geq\varepsilon
 \right\}
 \right|
 \leq n_{\varphi,\varepsilon}.
 \]

In particular, taking
 $g=b^\Omega$ for
 $b\in\mathcal U^y$, and using the
 definition of the canonical
 extension, we obtain
 \[
 \psi^p(x;b^\Omega)
 =
 \mu(\varphi(x;b)).
 \]

Consequently, condition (ii)
 holds, and hence $\mu$ is
 self-averaging over $M$.
 \end{proof}

As noted in the fact above, self-averaging only guarantees
 convergence for parameters from $\mathcal U$, and not
 necessarily for parameters from $\mathcal C$. The main
 motivation behind the definitions of $rgs$ and $irgs$
 in this paper is precisely to ensure convergence
 for parameters from $\mathcal C$ as well.

The next corollary summarizes all the previous observations
 and compares the different notions introduced in this
 section. Recall that by Fact~\ref{generic-fact-1},
 $fim$ and generic stability coincide for types
 (in both discrete and continuous logic).

\begin{Corollary}\label{comparing}
 For a definable global measure, the following implications
 hold.
 \begin{enumerate}
 \item[(i)] $fim\Rightarrow rgs\Rightarrow fam$.
 \item[(ii)]
 $fim\Rightarrow irgs\Rightarrow
 \mathrm{self\text{-}averaging}$.
 \item[(iii)] For types,
 $fim\Leftrightarrow irgs\Leftrightarrow
 \mathrm{self\text{-}averaging}$.
 \end{enumerate}
 \end{Corollary}

\begin{proof}
 (i) By Fact~\ref{r_mu-fact}, if $\mu$ is $fim$,
 then $r_\mu$ is also $fim$, and hence, by Theorem~\ref{main-irgs}, $\mu$ is $rgs$
 (and in fact also $irgs$).

Now suppose that $\mu$ is $rgs$ over $M$.
By Theorem~\ref{main-rgs}, it admits a
stable basic sequence $(\mu_i)$ over $M$.
As shown in Section~2, $\mu$ is definable
over $M$. Moreover, since each $\mu_i$
has finite support in $M$, it can be
approximated in total variation by an
averaging measure on a finite tuple from
$M$. Choosing these approximations with
errors tending to zero, we conclude that
$\mu$ is sequentially approximated over $M$.
Hence $\mu$ is $fam$ over $M$ by
\cite[Proposition~3.4]{Kyle-seq}.

(ii) Suppose first that $\mu$ is $fim$.
 By Fact~\ref{r_mu-fact}(iii),
 $r_\mu$ is generically stable,
 and hence $\mu$ is $irgs$.

Conversely, if $\mu$ is $irgs$,
 then $r_\mu$ is generically stable.
 By Fact~\ref{self-fact},
 $\mu$ is self-averaging over $M$.

(iii) For types, the implication $irgs$ (indeed $rgs$)
 $\Rightarrow fim$ is stated in Proposition~\ref{type-case}.
 Conversely, self-averaging implies generic stability
 (equivalently $fim$), as follows from the definition;
 see also Corollary~2.12 in \cite{CGH}.
 \end{proof}

 \medskip
 \noindent
 \textbf{Subsequent development.}
 After the first version of this paper appeared,
 Conant, Gannon, and Hanson \cite{CGH2026}
 established that, for definable global Keisler
 measures, the three notions $fim$, $irgs$,
 and self-averaging are equivalent. Thus,
 the implications in Corollary~\ref{comparing}(ii)
 are now known to be equivalences:
 \[
 fim\ \Longleftrightarrow\ irgs\\
 \Longleftrightarrow\ \mathrm{self\text{-}averaging}.
 \]
 We retain the original formulation and proof
 of Corollary~\ref{comparing} to reflect the
 development of these results.

The following example illustrates the distinction
between arbitrary random extensions of a Keisler
measure and its canonical random extension.
In particular, a measure may admit a generically
stable random extension which is different from
its canonical extension.

\begin{Example}\label{ex-noncanonical-extension}
Let $n\geq 2$, and let $a_1,\ldots,a_n$ be
distinct elements of a model $M$. Consider the
global Keisler measure
\[
\mu=\operatorname{Av}(a_1,\ldots,a_n)
=\frac{1}{n}\sum_{i=1}^n\delta_{a_i}.
\]
Since $\mu$ is finitely supported, it is a
$fim$ measure. Consequently, its canonical
random extension $r_\mu$ is generically stable
by \cite[Corollary~3.19]{CGH}.

Fix an atomless probability space
$(\Omega,\mathcal B,\mathcal P)$, and let
$B_1,\ldots,B_n$ be a measurable partition
of $\Omega$ with
\[
\mathcal P(B_i)=\frac{1}{n}
\qquad (1\leq i\leq n).
\]
Consider the element
\[
f=\sum_{i=1}^n a_i^{B_i}
\]
of the simple randomization $M_0^\Omega$,
and let
\[
p_f=\operatorname{tp}(f/\mathcal C)
\]
be its global type in the randomization,
where $\mathcal C$ is a sufficiently
saturated elementary extension of
$M^\Omega$.

Since $f\in M^\Omega$, the type $p_f$ is
realized in $M^\Omega$ and is therefore
generically stable over $M^\Omega$.

Moreover, for every classical formula
$\varphi(x;y)$ and every parameter
$b\in\mathcal U^y$, we have
\[
\begin{aligned}
\mathbb P^{p_f}
\llbracket\varphi(x;b^\Omega)\rrbracket
&=
\mathbb P\llbracket\varphi(f;b^\Omega)
\rrbracket\\
&=
\frac{1}{n}\sum_{i=1}^n\varphi(a_i;b)\\
&=\mu(\varphi(x;b)).
\end{aligned}
\]
Thus $p_f$ is a generically stable random
extension of $\mu$.

However, $p_f$ is not the canonical
extension $r_\mu$. Indeed, consider the
classical formula $\varphi(x;y)$ given
by $x=y$, and take the random parameter
$g=f$.

Since $p_f$ is realized by $f$, we have
\[
\mathbb P^{p_f}
\llbracket x=f\rrbracket=1.
\]
On the other hand, by the definition
of the canonical extension,
\[
\begin{aligned}
\mathbb P^{r_\mu}
\llbracket x=f\rrbracket
&=
\sum_{i=1}^n
\mathcal P(B_i)\mu(x=a_i)\\
&=
\sum_{i=1}^n
\frac{1}{n}\cdot\frac{1}{n}\\
&=\frac{1}{n}<1.
\end{aligned}
\]
Consequently,
\[
p_f\neq r_\mu.
\]

This example shows that generic stability
of a random extension does not force
that extension to be canonical.
In particular, the existence of a
generically stable random extension
need not imply that every random
extension of the same measure has
the canonical form.

It is important to note that this
example does not distinguish the
classes $rgs$ and $irgs$.
Indeed, $\mu$ is $fim$ and hence
$irgs$, so it admits an independent
stable basic sequence as well.

The example instead illustrates why
the distinction between arbitrary
generically stable random extensions
and the canonical extension is natural.
\end{Example}

The preceding example therefore leaves
open the question of whether the two
notions define different classes
of Keisler measures.

\begin{Question}
Does there exist an $rgs$ measure
which is not $irgs$?
\end{Question}

\section{Dependence and Convergence of Morley Sequences of $r_\mu$}

The concept of \textit{dependent} Keisler measure
 was introduced in \cite{K-dependent}, based on the
 Fremlin--Talagrand notion of stability for families
 of functions with respect to a measure. The latter
 notion is of great importance in measure theory,
 one of its key features being the characterization
 of Glivenko--Cantelli classes of measurable functions.

At first glance, this notion may appear somewhat
 technical and unnatural from a model-theoretic
 perspective. However, its natural properties and
 applications, partly presented in \cite{K-dependent},
 together with the comparison of dependent measures
 with well-behaved measures such as $fim$ carried
 out in the recent paper highlight the importance of
 studying this notion further in model theory.

The main result of this section is a divergence theorem
 for Morley sequences in canonical random extensions.
 More precisely, we prove that if a definable global
 Keisler measure $\mu$ is not dependent, then every
 Morley sequence in its canonical random extension
 $r_\mu$ is divergent
 (Theorem~\ref{nondependent-divergent}).
 As an immediate consequence, every $irgs$ measure
 is dependent (Corollary~\ref{irgs-dependent}).
 In view of the equivalence of $fim$ and $irgs$
 subsequently established by Conant, Gannon, and
 Hanson \cite{CGH2026}, this corollary also gives
 an alternative proof of the known fact that every
 $fim$ measure is dependent
 \cite[Proposition~30]{K-dependent}.

We start by recalling the definition of a dependent
 measure. First, we need to introduce notation.
 Let $\varphi(x;y)$ be an $L$-formula.
 For any $E\subseteq S_\varphi(\mathcal U)$ we write
 \[
 D_k(E,\varphi)=
 \left\{\bar p\in E^k:
 \begin{array}{l}
 \forall I\subseteq\{1,\dots,k\}\ 
 \exists b\in\mathcal U^y:\\
 \bigwedge_{i\in I}\varphi(p_i;b)=0,\\
 \bigwedge_{i\notin I}\varphi(p_i;b)=1
 \end{array}
 \right\}.
 \]

Recall that $\varphi(p_i;b)=1$ if
 $\varphi(x;b)\in p_i$ and $\varphi(p_i;b)=0$
 otherwise. For every formula $\varphi(x;y)$,
 the restriction map
 \[
 r_\varphi:S_x(\mathcal U)\to S_\varphi(\mathcal U)
 \]
 is a quotient map, where $S_\varphi(\mathcal U)$
 denotes the space of complete $\varphi$-types
 over $\mathcal U$. We set
 \[
 \mu_\varphi(E)
 =\mu(r_\varphi^{-1}(E))
 \]
 for any Borel subset
 $E\subseteq S_\varphi(\mathcal U)$.

\begin{Definition}[Dependent Measures]
 {\em
 Let $T$ be a complete theory, and $\mu_x$ a
 global measure. We say that $\mu$ is
 \emph{dependent}, if for any formula
 $\varphi(x;y)$ there is no measurable
 $E\subseteq S_\varphi(\mathcal U)$,
 $\mu_\varphi(E)>0$ (where $\mu_\varphi$
 is the restriction of $\mu$ to
 $S_\varphi(\mathcal U)$) such that for each $k$,
 \[
 (\mu_\varphi^k)D_k(E,\varphi)
 =(\mu_\varphi E)^k,
 \]
 where $\mu_\varphi^k$ is the usual product measure
 of $\mu_\varphi$.
 }
 \end{Definition}

\begin{Remark}
 Further explanations are given in Definition~7
 in \cite{K-dependent} and the remark following it.
 Note that our notation is simpler than there because
 we consider the sets $A,B$ there as the set of all
 parameters in $\mathcal U$, and since here we only
 work with global measures, we write $D_k(E,\varphi)$
 instead of $D_k(A,B,E,\varphi)$.
 Also in Remark~8(1) there we explained that the
 set $D_k(E,\varphi)$ is measurable.
 \end{Remark}

The following fact provides potential examples of
 dependent measures, clarifying the place of this
 concept in the subject (or context).
 Parts (i)--(iii) below appear in Proposition~11
 and Example~12 of \cite{K-dependent}, and part (iv)
 is Proposition~30 therein.

\begin{Fact}\label{dependent-ex}
 \begin{enumerate}
 \item[(i)] Every measure in any NIP theory is dependent.
 \item[(ii)] Every type in any arbitrary theory is dependent.
 \item[(iii)] For any sequence $(p_i)$ of global types
 and any sequence of positive real numbers $(r_i)$
 in $(0,1]$ with $\sum_i r_i=1$, the measure
 $\sum_i r_i\delta_{p_i}$ is a dependent measure,
 where $\delta_{p_i}$ is the Dirac measure corresponding
 to $p_i$.
 \item[(iv)] Every fim measure is dependent.
 \end{enumerate}
 \end{Fact}

Our aim is to establish the stronger divergence
 statement of Theorem~\ref{nondependent-divergent},
 from which Corollary~\ref{irgs-dependent} follows.
 The proof relies on a measure-theoretic characterization
 of non-dependence and its connection with the
 convergence of Morley sequences in the canonical
 random extension.

\medskip
 The following fact, which was not addressed in
 \cite{K-dependent}, is of great importance for
 the proof of the main theorem of this section.


\begin{Fact}[Talagrand]
 \label{Talagrand}
 Suppose that the global Keisler measure $\mu_x$
 is not dependent.
 Then there exist an $L$-formula $\varphi(x;y)$
 and $\mu_\varphi$-measurable functions
 $h_0,h_1:X\to[0,1]$, where
 $X=S_\varphi(\mathcal U)$,
 such that, with $\lambda=\mu_\varphi$,
 \[
 \Delta:=
 \int_X(h_1-h_0)\\d\lambda>0,
 \]
 and for each $k\geq1$ the set
 \[
 W_k=
 \left\{(p_1,\ldots,p_{2k})\in X^{2k}:
 \begin{array}{l}
 \text{there is }b\in\mathcal U^y
 \text{ such that for all }i\leq k,\\
 \varphi(p_{2i-1};b)\leq h_0(p_{2i-1})
 \text{ and}\\
 \varphi(p_{2i};b)\geq h_1(p_{2i})
 \end{array}
 \right\}
 \]
 has full outer product measure:
 \[
 (\lambda^{\otimes2k})^*(W_k)=1.
 \]
 \end{Fact}

\begin{proof}
 Apply Fremlin's Lemma~465L \cite{Fremlin}
 to the completed probability space $(X,\lambda)$
 and the non-stable family of $\{0,1\}$-valued
 functions
 \[
 \{p\mapsto\varphi(p;b):b\in\mathcal U^y\}.
 \]
 The completion is needed for the measurable
 threshold functions; the conclusion is stated
 in terms of outer measure. See
 Appendix~\ref{app:fremlin-reduction} for details.
 \end{proof}

\begin{Remark}
 The variable thresholds and outer measure in
 Fact~\ref{Talagrand} are essential.
 In general, non-dependence does not imply
 the full-shattering identity
 \[
 \lambda^{\otimes k}(D_k(X,\varphi))=1
 \]
 for every $k$.
 For example, a non-dependent measure may be mixed
 with a Dirac measure $\delta_q$ to obtain a
 non-dependent mixture
 $\mu=\frac12\rho+\frac12\delta_q$;
 nevertheless, for each $\psi$,
 \[
 \mu_\psi^{\otimes2}
 (D_2(S_\psi(\mathcal U),\psi))
 \leq\frac34,
 \]
 since two independent draws may both yield $q|_\psi$.
 \end{Remark}


\begin{Lemma}
 \label{lem:finite-threshold}
 Let $\varphi,h_0,h_1,\Delta$ be as in
 Fact~\ref{Talagrand}.
 Work in a sufficiently saturated model of
 $T^R$ containing $\mathcal U^\Omega$, and let
 $f_1,\ldots,f_{2k}$ be random elements whose
 joint classical $\varphi$-type distribution
 is $\lambda^{\otimes2k}$.
 Then for every $\eta>0$ there exists a random
 parameter $g$ such that, for each $1\leq i\leq k$,
 \[
 \mathbb P\llbracket\varphi(f_{2i};g)\rrbracket
 -
 \mathbb P\llbracket\varphi(f_{2i-1};g)\rrbracket
 \geq\Delta-2\eta.
 \]
 \end{Lemma}

\begin{proof}
 Put $X=S_\varphi(\mathcal U)$ and
 $\nu=\lambda^{\otimes2k}$.
 Replace the threshold functions by Borel
 representatives, changing nothing outside
 null sets. Since $\varphi(p;b)\in\{0,1\}$,
 \begin{align*}
 \varphi(p;b)\leq h_0(p)
 &\quad\Longleftrightarrow\quad
 h_0(p)=1\ \text{or}\ \varphi(p;b)=0,\\
 \varphi(p;b)\geq h_1(p)
 &\quad\Longleftrightarrow\quad
 h_1(p)=0\ \text{or}\ \varphi(p;b)=1.
 \end{align*}

Partition $X^{2k}$ into finitely many disjoint
 Borel cells $B_s$ according to which of the
 threshold constraints are vacuous.
 On each cell the non-vacuous conditions are
 a finite conjunction $\chi_s(\bar p;b)$
 of literals in $\varphi(p_j;b)$.

For each $b\in\mathcal U^y$ the set
 \[
 R_{s,b}=
 \{\bar p\in X^{2k}:\chi_s(\bar p;b)\}
 \]
 is clopen. Thus
 \[
 W_k=
 \bigcup_s\left(
 B_s\cap\bigcup_{b\in\mathcal U^y}R_{s,b}
 \right)
 \]
 is Borel. Fact~\ref{Talagrand} therefore gives
 $\nu(W_k)=1$.

By inner regularity, compactness, and the
 fact that $R_{s,b}$ are open, choose finitely
 many parameters $b_{s,j}\in\mathcal U^y$
 such that
 \[
 A:=
 \bigcup_{s,j}(B_s\cap R_{s,b_{s,j}}),
 \qquad
 \nu(A)>1-\frac{\eta}{2}.
 \]

Next approximate each Borel cell $B_s$ in
 $\nu$-measure by a clopen set
 $D_s\subseteq X^{2k}$, so that
 \[
 N:=\bigcup_s(B_s\mathbin\triangle D_s),
 \qquad
 \nu(N)\leq
 \sum_s\nu(B_s\mathbin\triangle D_s)
 <\frac{\eta}{2}.
 \]
 This approximation is possible because
 $X^{2k}$ is a compact zero-dimensional
 space and $\nu$ is a regular Borel measure.

Put
 \[
 C_{s,j}=D_s\cap R_{s,b_{s,j}}.
 \]
 These sets are clopen.
 Disjointize the finite family $(C_{s,j})$,
 in any fixed order, obtaining clopen
 sets $(H_{s,j})$ whose union is
 $\bigcup_{s,j}C_{s,j}$.
 Each $H_{s,j}$ defines an event in the
 randomization by evaluating its defining
 Boolean combination of classical
 $\varphi$-formulas at
 $f_1,\ldots,f_{2k}$ with constant
 parameters from $\mathcal U$.

Paste the constant random parameters
 $b_{s,j}^\Omega$ along these finitely
 many events, assigning an arbitrary
 constant $b_0^\Omega$ on the complementary
 event. The event/pasting axioms of $T^R$
 (or the perfect-witness property
 \cite[Theorem~2.7]{BK}) supply the
 resulting $g$ in the same random-element sort.

The finite pasting construction is detailed in
 Appendix~\ref{app:finite-selection}.

On the set $A\setminus N$, the selected
 parameter satisfies all threshold
 conditions; moreover
 \[
 \nu(A\setminus N)>1-\eta.
 \]
 Because the selected parameter depends
 only on a finite clopen partition, the
 joint-distribution hypothesis identifies
 each probability
 \[
 \mathbb P\llbracket\varphi(f_i;g)\rrbracket
 \]
 with the corresponding integral on $X^{2k}$.

The failure set has measure less than
 $\eta$, and all functions involved take
 values in $[0,1]$.
 Each coordinate has marginal $\lambda$,
 so for every $i\leq k$,
 \begin{align*}
 \mathbb P\llbracket
 \varphi(f_{2i-1};g)\rrbracket
 &\leq\int_Xh_0d\lambda+\eta,\\
 \mathbb P\llbracket
 \varphi(f_{2i};g)\rrbracket
 &\geq\int_Xh_1d\lambda-\eta.
 \end{align*}
 Subtracting proves the claim.
 \end{proof}

For later use we retain the original notation
 \[
 \Phi_k(x_1,\ldots,x_k):=
 \bigwedge_{I\subseteq\{1,\ldots,k\}}
 \exists y_I\\,
 \left(
 \bigwedge_{i\in I}\varphi(x_i;y_I)
 \ \land\\
 \bigwedge_{i\notin I}
 \neg\varphi(x_i;y_I)
 \right).
 \]
 The formula $\Phi_k$ is not needed in the
 proof of Theorem~\ref{nondependent-divergent},
 but is retained for Section~5.

\medskip
 We are now ready for the main theorem of this section.


\begin{Theorem}\label{nondependent-divergent}
 Suppose that the global measure $\mu$ is definable but not dependent.
 Then every Morley sequence in its canonical random type $r_\mu$
 (over $\mathcal U^\Omega$) is divergent.
 \end{Theorem}

\begin{proof}
 Let $p=r_\mu$, and fix an arbitrary Morley sequence
 $(f_i)_{i\geq 1}$ in $p$ over $\mathcal U^\Omega$.
 We work in a sufficiently saturated ambient model
 of $T^R$, and suppose, towards a contradiction,
 that $(f_i)$ is convergent.

By non-dependence and Fact~\ref{Talagrand}, choose
 an $L$-formula $\varphi(x;y)$, measurable functions
 $h_0,h_1$, and a number
 \[
 \Delta=\int(h_1-h_0)\\,d\mu_\varphi>0.
 \]
 Put $\lambda=\mu_\varphi$ and let
 \[
 \psi(x;y)=\mathbb P\llbracket\varphi(x;y)\rrbracket.
 \]

For every $k\geq 1$, the joint distribution of the
 classical $\varphi$-types of
 $f_1,\ldots,f_{2k}$ is $\lambda^{\otimes 2k}$.
 Indeed, by Fact~\ref{r_mu-fact}(ii),
 \[
 p^{(2k)}=r_{\mu^{(2k)}},
 \]
 The product law is justified in
 Appendix~\ref{app:morley-law}.

Fix $\eta=\Delta/8$. By
 Lemma~\ref{lem:finite-threshold}, for every
 $k\geq 1$ there is a random parameter $g_k$
 such that
 \[
 \psi(f_{2i};g_k)-\psi(f_{2i-1};g_k)
 \geq \Delta-2\eta=\frac{3\Delta}{4}
 \qquad (1\leq i\leq k).
 \tag{1}
 \]

Consider the following partial type in the
 parameter variable $y$:
 \[
 \Sigma(y)=\left\{
 \psi(f_{2i};y)-\psi(f_{2i-1};y)
 \geq \frac{\Delta}{2}:i\geq 1
 \right\}.
 \]
 Every finite subset of $\Sigma$ is realized
 by $g_k$ for sufficiently large $k$, by (1).
 By compactness and saturation (see
 Appendix~\ref{app:compactness-detail}), some $g$
 realizes $\Sigma$. Thus
 \[
 \psi(f_{2i};g)-\psi(f_{2i-1};g)
 \geq \frac{\Delta}{2}>0
 \qquad\text{for every }i\geq 1.
 \tag{2}
 \]

By the assumed convergence, the real
 sequence $(\psi(f_i;g))_{i\geq 1}$ converges.
 In particular,
 \[
 \lim_{i\to\infty}
 \bigl(\psi(f_{2i};g)-\psi(f_{2i-1};g)\bigr)=0,
 \]
 contradicting (2). Hence $(f_i)$ is divergent.
 Since the Morley sequence was arbitrary, the
 conclusion follows.
 \end{proof}

\begin{Corollary}\label{irgs-dependent}
 Every $irgs$ measure is dependent.
 \end{Corollary}

\begin{proof}
 If $\mu$ is $irgs$, then $r_\mu$ is generically stable by
 Definition~\ref{rgs-irgs}.
 Every Morley sequence in $r_\mu$ therefore converges by
 Fact~\ref{generic-fact-1}.
 By Theorem~\ref{nondependent-divergent}, if $\mu$ were not
 dependent, then some (indeed every) Morley sequence in
 $r_\mu$ would diverge, a contradiction.
 Hence $\mu$ is dependent.
 \end{proof}
 In view of the equivalence of $fim$ and $irgs$
 established by Conant, Gannon, and Hanson
 \cite[Theorem~1.1]{CGH2026}, this corollary is
 equivalent to the result that every $fim$
 measure is dependent \cite[Proposition~30]{K-dependent},
 albeit with a different proof.

\begin{Corollary}\label{cor:dfs-convergent-random}
Let $\mu$ be a global Keisler measure which is
definable over and finitely satisfiable in
a small model $M$.
Suppose that some Morley sequence in its
canonical random extension $r_\mu$ over
$M^\Omega$ is convergent as a sequence
of types over the monster model of $T^R$.
Then $\mu$ is $irgs$ over $M$.
\end{Corollary}

\begin{proof}
Put $p=r_\mu$. Since all Morley sequences
of $p$ over $M^\Omega$ have the same type
over $M^\Omega$, convergence of one such
sequence implies convergence of all of them.
Indeed, by homogeneity, a divergent Morley
sequence would transfer an oscillating
formula and its parameter to the given
convergent sequence, a contradiction.

In particular, every Morley sequence of
$p$ over $\mathcal U^\Omega$ is convergent.
Hence, by
Theorem~\ref{nondependent-divergent},
$\mu$ is dependent.

Since $\mu$ is definable, finitely
satisfiable, and dependent, it follows
from \cite[Corollary~24]{K-dependent}
that $\mu$ is symmetric. Thus, for every
$n$, the Morley power $\mu^{(n)}$ is
invariant under permutations of its
coordinates.

By Fact~\ref{r_mu-fact}(ii),
\[
(r_\mu)^{(n)}=r_{\mu^{(n)}}.
\]
Consequently, every Morley sequence of
$p$ over $M^\Omega$ is totally
indiscernible.

We have therefore shown that every
Morley sequence of $p$ over $M^\Omega$
is both convergent and totally
indiscernible. By the characterization
of generic stability in
\cite[Remark~2.9(iv)]{K-generic},
$p=r_\mu$ is generically stable over
$M^\Omega$.

Therefore, by the definition of $irgs$,
$\mu$ is $irgs$ over $M$.
\end{proof}

By the recent result of Conant, Gannon,
and Hanson \cite{CGH2026}, $irgs$ and
$fim$ are equivalent for definable
global Keisler measures. Consequently,
the measure $\mu$ in the preceding
corollary is also $fim$.

\begin{Remark}[Necessity of the assumptions]
Neither finite satisfiability nor convergence
can be omitted from
Corollary~\ref{cor:dfs-convergent-random}.

\emph{(i) Finite satisfiability.}
Let $p$ be the global type at $+\infty$
in DLO. Then $p$ is definable but not
finitely satisfiable in any small model.
Since the randomization of DLO has NIP,
every Morley sequence of $r_p$ converges.
However, $p$ is not generically stable
and hence not $irgs$, by
Proposition~\ref{type-case}.

\emph{(ii) Convergence.}
Conant, Gannon, and Hanson constructed
a global $dfs$ type $q$ which is not $fam$
\cite[Corollary~7.12]{CGH-wild}.
Since $fim$ implies $fam$, the type $q$
is not $fim$, and hence not $irgs$,
by Proposition~\ref{type-case}.
Thus the convergence assumption is
also essential.
\end{Remark}


\begin{Remark}[The role of the canonical random extension]
 The proof of Theorem~\ref{nondependent-divergent}
 uses the canonical random extension $r_\mu$.

To see where this matters, suppose that
 $p_\mu$ is a generically stable random
 type extending a global non-dependent
 measure $\mu$.
 Fix $\varphi,h_0,h_1,\Delta,W_k$
 as in Fact~\ref{Talagrand}, and let
 $\nu_{2k}$ be the joint classical
 $\varphi$-type distribution of the
 first $2k$ coordinates of
 $(p_\mu)^{(2k)}$.
 Each coordinate has marginal
 $\lambda=\mu_\varphi$.

If there is a fixed $0<\eta<\Delta/4$
 such that
 \[
 \nu_{2k}(W_k)>1-\eta
 \qquad\text{for every }k\geq1,
 \]
 then, by the finite-threshold-selection
 argument of Lemma~\ref{lem:finite-threshold},
 for any sufficiently small
 $\varepsilon>0$ and every $k$
 there is a random parameter $g_k$
 for which, simultaneously for
 $1\leq i\leq k$,
 \[
 \mathbb P\llbracket
 \varphi(f_{2i};g_k)\rrbracket
 -
 \mathbb P\llbracket
 \varphi(f_{2i-1};g_k)\rrbracket
 \geq\Delta-2(\eta+\varepsilon)>0.
 \]

Indeed, the finite clopen approximation
 in that lemma only requires that
 $\nu_{2k}(W_k)>1-\eta$, and the common
 coordinate marginals identify the
 threshold integrals with those of
 $\lambda$.

Compactness then contradicts generic
 stability of $p_\mu$.

For the canonical extension $r_\mu$,
 the required hypothesis holds with
 $\nu_{2k}=\lambda^{\otimes2k}$
 by the coordinate-product identity.

For a noncanonical extension $p_\mu$,
 it does not follow merely from agreement
 with $\mu$ on constant random parameters.
 Thus the question whether every $rgs$
 measure is dependent remains open.
 \end{Remark}

\begin{Question}
 Is every $rgs$ measure dependent?
 \end{Question}

\begin{Corollary}\label{cor}
 Let $\mu$ be a global $irgs$ measure over
 $M$, and $\lambda$ a global measure which
 is $M$-finitely satisfied. Then
 \begin{enumerate}
 \item[(i)] The Morley product
 $\lambda\otimes\mu$ is well-defined,
 and $\mu\otimes\lambda
 =\lambda\otimes\mu$.
 \item[(ii)] $\mu$ is symmetric; that is,
 for every permutation $\sigma$ of
 $\{1,\dots,n\}$,
 \[
 \mu_{x_1}\otimes\cdots\otimes\mu_{x_n}
 =
 \mu_{\sigma(x_1)}\otimes\cdots
 \otimes\mu_{\sigma(x_n)}.
 \]
 \end{enumerate}
 \end{Corollary}

\begin{proof}
 (i): This follows from  Corollary~\ref{irgs-dependent} above and
 Proposition~16 (and the definition
 following it) and Theorem~21 in
 \cite{K-dependent}.

(ii): This follows from (i).
 See also Corollary~24 in
 \cite{K-dependent}.
 \end{proof}

\section{Model Theoretic Events and $irgs$}

In this section, we study a notion that is
 inherently related to the previous discussion,
 but from a different conceptual perspective,
 as investigated in \cite{GH}.

In that work, the Morley measure
 $\mu^{(\omega)}$ associated with a
 measure $\mu$ (for which Morley products
 are well-defined) is studied. This
 measure is denoted by $\mathbb P_\mu$.

In Section~6 of that paper, model-theoretic
 events corresponding to the properties OP,
 IP, and SOP in the dividing patterns of the
 theory are introduced. Since these events
 satisfy a 0--1 law with respect to
 $\mathbb P_\mu$, it is shown in
 Theorem~6.7 of \cite{GH} that if
 $\mu$ is fim, then the measure of
 the model-theoretic non-tameness events
 under $\mathbb P_\mu$ is zero.

In this section, we revisit the vanishing of
 model-theoretic instability events from the
 perspective of canonical random extensions.
 Gannon and Hanson proved in
 \cite[Theorem~6.7]{GH} that these events have
 probability zero for $fim$ measures.
 Our approach is based on generic stability
 of the canonical random type and uniform
 convergence in the randomization. It is
 closely related to the divergence phenomenon
 investigated in Section~4.

After the first version of this paper appeared,
 Conant, Gannon, and Hanson established
 in \cite[Theorem~1.1]{CGH2026} that the
 notions of $fim$ and $irgs$ are equivalent.
 Consequently, Theorem~\ref{irgs-tame-event}
 does not extend the class of measures covered
 by \cite[Theorem~6.7]{GH}. Its interest lies
 instead in providing a different approach
 to the vanishing phenomenon, formulated
 in terms of random types and convergence.

We also investigate the relationship between
 the order and independence events under
 symmetry, and examine examples illustrating
 the limitations of measure-theoretic
 dependence in controlling these events.

The notation of this section follows that
 of \cite{GH} to facilitate comparison and
 adaptation; therefore, instead of
 $\mu^{(\omega)}$ and $\Phi_k$ used in
 previous sections, we use $\mathbb P_\mu$
 and $I_k^\varphi$.

Also, some of the technique and definitions
 have been omitted; therefore, for a more
 detailed definition and additional explanations,
 refer to the original reference.

In what follows, for each formula
 $\varphi(\bar x)$, by $[\varphi(\bar x)]$
 we mean the corresponding event in
 $S_{\mathbf{x}}(\mathcal U)$, where
 $\mathbf{x}$ is a countable sequence of
 variables $x_1,x_2,\ldots$.

\begin{Definition}[\cite{GH}, Def.~6.3]
 We define the following model-theoretic events:
 \begin{itemize}
 \item
 $\mathbf O^\varphi
 =\bigcap_{k=1}^\infty
 [O_k^\varphi(x_1,\ldots,x_k)]$:
 where $O_k^\varphi(\bar x)$ expresses
 that there exists an order of length
 $k$ for $\varphi(x;y)$;
 \item
 $\mathbf I^\varphi
 =\bigcap_{k=1}^\infty
 [I_k^\varphi(x_1,\ldots,x_k)]$:
 where $I_k^\varphi(\bar x)$ expresses
 that there exists a shattering of
 length $k$ for $\varphi(x;y)$;
 \item
 $\mathbf L^\varphi
 =\bigcap_{k=1}^\infty
 [L_k^\varphi(x_1,\ldots,x_k)]$:
 where
 $L_k^\varphi(\bar x)
 =\forall y(\varphi(x_1;y)
 \subsetneq\ldots\subsetneq
 \varphi(x_k;y))$.
 \end{itemize}
 We also define the global instability
 events by taking the union over all
 formulas in the language. More precisely, let
 \[
 \mathbf O:=
 \bigcup_{\varphi\in L}\mathbf O^\varphi,
 \quad
 \mathbf I:=
 \bigcup_{\varphi\in L}\mathbf I^\varphi,
 \quad
 \mathbf L:=
 \bigcup_{\varphi\in L}\mathbf L^\varphi.
 \]
 \end{Definition}

\begin{Remark}
 
 Note that the event $\mathbf I^\varphi$
 above is precisely equal to
 $\bigcap_{k=1}^\infty[\Phi_k]$,
 where $\Phi_k$ is the formula
 associated to $\varphi$ defined
 immediately before
 Theorem~\ref{nondependent-divergent}.
 \end{Remark}

\begin{Fact}[\cite{GH}, Lemma~6.5]\label{0-1}
Let $\mu$ be a global Keisler measure
definable over a small model $M$.
Then, for every classical formula
$\varphi(x;y)$,
\[
\mathbb P_\mu(E)\in\{0,1\}
\]
whenever
\[
E\in\{
\mathbf O^\varphi,
\mathbf I^\varphi,
\mathbf L^\varphi
\}.
\]
\end{Fact}

The following theorem is essentially an
 adaptation of the ideas underlying
 Theorem~\ref{nondependent-divergent}.

\begin{Theorem}\label{irgs-tame-event}
 Let $\mu$ be an $irgs$ Keisler measure.
 Then
 \[
 \mathbb P_\mu(\mathbf O^\varphi)
 =
 \mathbb P_\mu(\mathbf I^\varphi)
 =
 \mathbb P_\mu(\mathbf L^\varphi)
 =0.
 \]
 Consequently,
 \[
 \mathbb P_\mu(\mathbf O)
 =
 \mathbb P_\mu(\mathbf I)
 =
 \mathbb P_\mu(\mathbf L)
 =0.
 \]
 \end{Theorem}

\begin{proof}
 It suffices to prove that
 \[
 \mathbb P_\mu(\mathbf O^\varphi)=0
 \]
 for every formula $\varphi(x;y)$.
 Suppose, towards a contradiction, that
 \[
 \mathbb P_\mu(\mathbf O^\varphi)>0.
 \]
 By Fact~\ref{0-1}, we have
 \[
 \mathbb P_\mu(\mathbf O^\varphi)=1.
 \]
 Consequently, for every $k\geq 1$,
 \[
 \mu^{(k)}(O_k^\varphi)=1.
 \tag{1}
 \]

Let $p=r_\mu$. Since $\mu$ is $irgs$,
 $p$ is generically stable.
 Fix a Morley sequence
 $(f_i)_{i<\omega}$ in $p$.

Consider the formula of $T^R$
 \[
 \psi(x;y)=
 \mathbb P\llbracket\varphi(x;y)\rrbracket.
 \]
 By Fact~\ref{generic-fact-1}(iii), applied
 with $\varepsilon=1/4$, there exists
 $n<\omega$ such that, for every
 $g\in\mathcal C^y$,
 \[
 \left|
 \left\{
 i<\omega:
 \left|\psi(f_i;g)-\psi^p(x;g)\right|
 \geq\frac14
 \right\}
 \right|
 \leq n.
 \tag{2}
 \]

Choose $k=2n+2$.
 By Fact~\ref{r_mu-fact}(ii),
 \[
 p^{(k)}=r_{\mu^{(k)}}.
 \]
 Combining this with (1), we obtain
 \[
 p^{(k)}
 \left(
 \mathbb P\llbracket
 O_k^\varphi(x_1,\ldots,x_k)
 \rrbracket
 \right)=1.
 \]
 Since $(f_i)_{i<\omega}$ is a
 Morley sequence in $p$, it follows that
 \[
 \mathbb P\llbracket
 O_k^\varphi(f_0,\ldots,f_{k-1})
 \rrbracket=1.
 \tag{3}
 \]

Write the order-pattern formula as
 \[
 O_k^\varphi(\bar x)
 =
 \exists y_0,\ldots,y_{k-1}\\,
 \theta_k(\bar x;\bar y),
 \]
 where
 \[
 \theta_k(\bar x;\bar y)
 =
 \bigwedge_{0\leq i\leq j<k}
 \varphi(x_i;y_j)
 \ \wedge\\
 \bigwedge_{0\leq j<i<k}
 \neg\varphi(x_i;y_j).
 \]

By (3) and the Fullness Axioms
 of randomization, we can choose
 $g_0,\ldots,g_{k-1}\in\mathcal C^y$
 such that
 \[
 \mathbb P\llbracket
 \theta_k(f_0,\ldots,f_{k-1};
 g_0,\ldots,g_{k-1})
 \rrbracket
 >\frac78.
 \]
 In particular,
 \[
 \psi(f_i;g_j)
 \begin{cases}
 >\frac78,& i\leq j,\\[2mm]
 <\frac18,& i>j.
 \end{cases}
 \tag{4}
 \]

Now take $j=n$ and put
 \[
 \alpha=\psi^p(x;g_n)\in[0,1].
 \]

If $\alpha\leq\frac12$, then (4)
 implies that, for every $i\leq n$,
 \[
 |\psi(f_i;g_n)-\alpha|
 >\frac38>\frac14.
 \]
 Thus the set appearing in (2)
 contains at least $n+1$ indices,
 contrary to (2).

If $\alpha>\frac12$, then for
 every $n<i<k$, equation (4) gives
 \[
 |\psi(f_i;g_n)-\alpha|
 >\frac38>\frac14.
 \]
 Again, since
 \[
 k-(n+1)=n+1,
 \]
 the set in (2) contains at least
 $n+1$ indices, a contradiction.

Therefore
 \[
 \mathbb P_\mu(\mathbf O^\varphi)=0.
 \]

Since
 \[
 \mathbf I^\varphi\subseteq
 \mathbf O^\varphi,
 \qquad
 \mathbf L^\varphi\subseteq
 \mathbf O^\varphi,
 \]
 we conclude that
 \[
 \mathbb P_\mu(\mathbf I^\varphi)
 =
 \mathbb P_\mu(\mathbf L^\varphi)
 =0.
 \]

Finally, as $L$ is countable,
 the corresponding global events
 are countable unions of null events.
 Hence
 \[
 \mathbb P_\mu(\mathbf O)
 =
 \mathbb P_\mu(\mathbf I)
 =
 \mathbb P_\mu(\mathbf L)
 =0.
 \]
 \end{proof}

\begin{Remark}
 The corollary 6.9 of \cite{GH},
 which is an eventual version of the
 above theorem, also holds for
 $irgs$ measures by a similar argument.
 \end{Remark}

The notion of a symmetric measure
 (i.e., a measure such that
 $\mu_x\otimes\mu_y=\mu_y\otimes\mu_x$)
 was introduced and studied in
 Corollary~\ref{cor} for $irgs$ measures.

The following proposition shows the
 relationship between symmetry and
 the event $\mathbf I$ with the event
 $\mathbf O$.

\begin{Observation}\label{observation}
 Let $\mu$ be a definable global Keisler measure
 whose finite Morley powers are invariant under
 permutations of coordinates. Then, for every
 $L$-formula $\varphi(x;y)$,
 \[
 \mathbb P_\mu(\mathbf O^\varphi)
 =
 \mathbb P_\mu(\mathbf I^\varphi).
 \]
 Consequently, if $L$ is countable, then
 \[
 \mathbb P_\mu(\mathbf O)
 =
 \mathbb P_\mu(\mathbf I).
 \]
 \end{Observation}

\begin{proof}
 Since $\mathbf I^\varphi\subseteq\mathbf O^\varphi$,
 it suffices to prove that
 \[
 \mathbb P_\mu(\mathbf O^\varphi)=1
 \quad\Longrightarrow\quad
 \mathbb P_\mu(\mathbf I^\varphi)=1.
 \]
 Indeed, by Fact~\ref{0-1}, both probabilities
 belong to $\{0,1\}$.

Suppose that
 $\mathbb P_\mu(\mathbf O^\varphi)=1$.
 Fix $n\geq 1$.
 For every permutation $\sigma$ of
 $\{1,\ldots,n+1\}$, symmetry gives
 \[
 \mathbb P_\mu\bigl(
 [O_{n+1}^\varphi
 (x_{\sigma(1)},\ldots,x_{\sigma(n+1)})]
 \bigr)=1.
 \]
 Since there are only finitely many such
 permutations, their corresponding events
 have an intersection of probability one.

We claim that every tuple in this intersection
 satisfies $I_n^\varphi(x_1,\ldots,x_n)$.

Indeed, let $J\subseteq\{1,\ldots,n\}$.
 If $J\neq\varnothing$, choose a permutation
 placing precisely the coordinates indexed
 by $J$ first, followed by $x_{n+1}$.
 The order pattern for this permutation
 provides a parameter realizing $\varphi$
 exactly on the coordinates indexed by $J$.

If $J=\varnothing$, choose a permutation
 placing $x_{n+1}$ first. The first parameter
 in the corresponding order pattern makes
 $\varphi$ false on all of
 $x_1,\ldots,x_n$.

Thus every Boolean pattern on
 $x_1,\ldots,x_n$ is realized, proving
 \[
 \mathbb P_\mu([I_n^\varphi])=1.
 \]
 Since $n$ was arbitrary,
 \[
 \mathbb P_\mu(\mathbf I^\varphi)=1.
 \]

This proves the first assertion.
 For the second, use the countability of
 $L$ and the definitions
 \[
 \mathbf O=\bigcup_{\varphi\in L}
 \mathbf O^\varphi,
 \qquad
 \mathbf I=\bigcup_{\varphi\in L}
 \mathbf I^\varphi.
 \]
 \end{proof}



The next result shows that, for an
 important class of measures,
 $\mathbb P_\mu(\mathbf O)$ and
 $\mathbb P_\mu(\mathbf I)$ coincide.


\begin{Corollary}\label{cor:dependent-dfs-O-I}
 Let $T$ be a complete theory in a countable language $L$,
 and let $\mu$ be a global Keisler measure which is
 definable over and finitely satisfiable in a small model $M$.
 If $\mu$ is dependent, then
 \[
 \mathbb{P}_{\mu}(\mathbf{O})
 =
 \mathbb{P}_{\mu}(\mathbf{I}).
 \]
 \end{Corollary}

\begin{proof}
 Since $\mu$ is dependent, definable over $M$, and finitely
 satisfiable in $M$, it follows from
 \cite[Corollary~24]{K-dependent} that $\mu$ is symmetric.
 In particular, all finite Morley powers of $\mu$ are
 invariant under permutations of coordinates. Thus, the desired result follows from Observation~\ref{observation}.




 \end{proof}

\begin{Question}\label{O=I}
 Suppose that $\mu$ is a global measure
 definable over and finitely satisfiable
 in a small model. Must
 \[
 \mathbb P_\mu(\mathbf O)
 =
 \mathbb P_\mu(\mathbf I)?
 \]
 \end{Question}

\medskip

In \cite{GH}, three well-known examples
 of Keisler measures are presented and
 their properties are analyzed.
 Here, we only discuss the properties of
 two of these examples and indicate
 whether they are dependent or not.
 The third example (Example~6.11 therein)
 will be discussed in more detail
 in the final remark.

In Example~6.10 therein, a Keisler
 measure is presented in the theory
 of the Rado graph which is definable
 and symmetric, and satisfies
 \[
 \mathbb P_\mu(\mathbf I)
 =
 \mathbb P_\mu(\mathbf O)=1,
 \qquad
 \mathbb P_\mu(\mathbf L)=0.
 \]
 It can be shown that this measure is
 not dependent. The idea of this example is used in
 Remark~\ref{Example/Remark}(ii) below.

In Example~6.12 therein, a Dirac
 measure is constructed in an NIP
 theory which is definable but not $fim$.
 In this case,
 \[
 \mathbb P_\mu(\mathbf O)
 =
 \mathbb P_\mu(\mathbf L)=1,
 \qquad
 \mathbb P_\mu(\mathbf I)=0.
 \]
 This measure is dependent since it
 is a type, and it is not symmetric
 because the language contains the
 relation $<$.

We now describe the third example
 in more detail.

\begin{Remark}[Example/Remark] \label{Example/Remark}

 Dependence does not imply
 $\mathbb P_\mu(\mathbf I)=0$.  
 Conversely, the proof of
 Theorem~\ref{nondependent-divergent}
 does not imply that every
 non-dependent measure satisfies
 $\mathbb P_\mu(\mathbf I)=1$. See (ii) below.

(i) For this purpose, consider the theory
 of the triangle-free random graph
 (the Henson graph) in the language
 $L=\{R\}$, together with the complete
 type
 \[
 p=\{\neg R(x,a):a\in\mathcal U\}.
 \]
 As mentioned in Example~6.11 of
 \cite{GH}, the Dirac measure
 $\mu=\delta_p$ is a $fam$, and
 since every Morley sequence of it
 is shattered, we conclude that
 \[
 \mathbb P_\mu(\mathbf I)
 =
 \mathbb P_\mu(\mathbf O)=1.
 \]
 Also, since the theory is NSOP,
 we have
 \[
 \mathbb P_\mu(\mathbf L)=0.
 \]
 Nevertheless, since $\mu$ is a type,
 it is dependent
 (see Fact~\ref{dependent-ex}).

This example is consistent with
 Observation~\ref{observation}, since
 $\mu$ is symmetric and
 \[
 \mathbb P_\mu(\mathbf O)
 =
 \mathbb P_\mu(\mathbf I)=1.
 \]
 It also shows that dependence alone
 does not imply the vanishing of the
 independence event.

 (ii) Let $T$ be the theory of the Rado random graph
 in the language $L=\{R\}$, where $R$ denotes
 the adjacency relation.

Consider the global Keisler measure $\rho$
 defined as follows. For any pairwise distinct
 parameters $b_1,\ldots,b_n\in\mathcal U$ and
 any $\epsilon_1,\ldots,\epsilon_n\in\{0,1\}$,
 set
 \[
 \rho\left(
 \bigwedge_{i=1}^{n}
 R(x,b_i)^{\epsilon_i}
 \right)=2^{-n},
 \]
 where $R(x,b)^1=R(x,b)$ and
 $R(x,b)^0=\neg R(x,b)$.
 We also require
 \[
 \rho(x=b)=0
 \qquad(b\in\mathcal U).
 \]
 By quantifier elimination for the random graph,
 these prescriptions determine a global Keisler
 measure.  
 Moreover, $\rho$ is definable: by quantifier
 elimination for the random graph, the
 measure of any formula $\psi(x;\bar b)$
 depends only on the quantifier-free type
 of $\bar b$.

We first verify that $\rho$ is non-dependent.
 Let
 \[
 \varphi(x;y)=R(x,y),
 \qquad
 X=S_\varphi(\mathcal U).
 \]
 Fix a sequence $(b_j)_{j<\omega}$ of pairwise
 distinct elements of $\mathcal U$.

If $p_1,\ldots,p_k$ are sampled independently
 according to $\rho_\varphi$, then the random
 vectors
 \[
 \bigl(\varphi(p_1;b_j),\ldots,
 \varphi(p_k;b_j)\bigr),
 \qquad j<\omega,
 \]
 are independent and uniformly distributed
 on $\{0,1\}^k$.

Consequently, with product probability one,
 every pattern in $\{0,1\}^k$ occurs among
 these vectors. Hence
 \[
 \rho_\varphi^{\otimes k}
 \bigl(D_k(X,\varphi)\bigr)=1
 \qquad(k\geq1).
 \]
 Thus $\rho$ is non-dependent.

Now fix $a\in\mathcal U$ and define
 \[
 \mu=\frac12\rho+\frac12\delta_a,
 \]
 where $\delta_a$ is the Dirac measure
 concentrated at $a$.

The measure $\mu$ is definable over any
 small model containing $a$.

We claim that $\mu$ is still non-dependent.

Let
 \[
 p_a=\operatorname{tp}_\varphi(a/\mathcal U),
 \qquad
 E=X\setminus\{p_a\}.
 \]
 Since $\rho_\varphi$ is atomless,
 \[
 \rho_\varphi(E)=1,
 \qquad
 \mu_\varphi(E)=\frac12.
 \]
 Moreover, the restriction of $\mu_\varphi$
 to $E$ is precisely
 $\frac12\rho_\varphi|_E$.

Therefore, for every $k\geq1$,
 \[
 \begin{aligned}
 \mu_\varphi^{\otimes k}
 \bigl(D_k(E,\varphi)\bigr)
 &=
 2^{-k}
 \rho_\varphi^{\otimes k}
 \bigl(D_k(E,\varphi)\bigr)\\
 &=2^{-k}\\
 &=\bigl(\mu_\varphi(E)\bigr)^k.
 \end{aligned}
 \]
 This proves that $\mu$ is non-dependent.

We next show that nevertheless
 \[\mathbb P_\mu(\mathbf I)=0.
 \]

Recall that $\mathbb P_\mu$ is the measure
 associated with the infinite Morley product
 $\mu^{(\omega)}$.

For each coordinate $x_i$, the event
 $x_i=a$ has probability $1/2$.
 Furthermore, these events are independent
 under the Morley product, since the defining
 formula $x=a$ has constant $\mu$-measure
 $1/2$.

Thus, for each $n\geq2$,
 \[
 \mathbb P_\mu
 \left(
 \text{at most one of }
 x_1,\ldots,x_n
 \text{ equals }a
 \right)
 =
 \frac{n+1}{2^n}.
 \]

On the other hand, if two coordinates
 $x_i,x_j$ are equal to $a$, no formula
 $\psi(x;y)$ can shatter a tuple containing
 both coordinates. Indeed, the two equal
 coordinates cannot be assigned different
 truth values by the same formula and
 parameter.

Consequently, the independence event
 $\mathbf I$ is contained in the event
 that at most one coordinate of the
 infinite sequence equals $a$.

It follows that, for every $n\geq2$,
 \[
 \mathbb P_\mu(\mathbf I)
 \leq \frac{n+1}{2^n}.
 \]
 Letting $n\to\infty$, we obtain
 \[
 \mathbb P_\mu(\mathbf I)=0.
 \]

We have therefore constructed a definable
 global Keisler measure satisfying
 \[
 \mu\text{ is non-dependent},
 \qquad
 \mathbb P_\mu(\mathbf I)=0.
 \]

In particular, the implication
 \[
 \mathbb P_\mu(\mathbf I)=0
 \quad\Longrightarrow\quad
 \mu\text{ is dependent}
 \]
 fails in general, even for definable
 global measures.

The example does not provide a counterexample
 to Question~\ref{O=I}, which additionally
 assumes finite satisfiability in a small
 model. Indeed, the measure $\mu$ constructed
 above is not finitely satisfiable in any
 small elementary submodel.
 \end{Remark}

\paragraph{Note added after the initial preprint.} 
 The first version of this paper appeared on
 arXiv on May 15, 2026. Subsequently, Conant,
 Gannon, and Hanson \cite{CGH2026} established
 that, for Borel-definable global Keisler
 measures, the notions of $fim$, generic
 stability of the canonical random extension,
 and self-averaging are equivalent.
 In the terminology of the present paper,
 their result shows, in particular, that
 $irgs$ and $fim$ coincide. Consequently,
 some of the implications established here
 are now known to be equivalences.
 We have retained our original approach and
 independent arguments, particularly the
 connection between non-dependence and
 divergence of Morley sequences in the
 canonical random extension, as they provide
 a different perspective on these notions. 
 We also note that Conant, Gannon, and Hanson
 discuss the notion of $rgs$ introduced here
 in Section 6.2 of their paper, including
 the question of its relationship with
 $irgs$.

\section*{Acknowledgments}

I would like to thank the School
 of Mathematics, Institute for Research
 in Fundamental Sciences (IPM),
 P. O. Box 19395-5746, Tehran, Iran,
 for their support. This research
 was in part supported by a grant
 from IPM (No. 1405030024).

\section*{Use of AI Tools}

During the preparation and revision of the final version of this manuscript, ChatGPT (OpenAI) was used to assist with language editing, improving the clarity and organization of the text, identifying potential mathematical gaps, and developing and refining the presentation of certain arguments and proofs, including technical details in the appendix. The author retains full responsibility for the mathematical claims, proofs, accuracy, and scientific content of the final manuscript.

\appendix
\section{Technical details for the divergence theorem}
\label{app:divergence-details}

This appendix records the measure-theoretic and
randomization-theoretic details used in
Fact~\ref{Talagrand}, Lemma~\ref{lem:finite-threshold},
and Theorem~\ref{nondependent-divergent}.
The notation and hypotheses are those of Section~4.

\subsection{From non-dependence to Fremlin's lemma}
\label{app:fremlin-reduction}

Let $X=S_\varphi(\mathcal U)$ and
$\lambda=\mu_\varphi$, and consider the family
\[
\mathcal A_\varphi
 =\{F_b:X\to\{0,1\}:b\in\mathcal U^y\},
\qquad F_b(p)=\varphi(p;b).
\]
Each $F_b$ is continuous in the Stone topology.
We explain why non-dependence of $\mu$ gives precisely
the non-stability hypothesis required by
\cite[Lemma~465L]{Fremlin}.

By non-dependence, for some $\varphi$ there is a
$\lambda$-measurable set $E\subseteq X$ with
$\lambda(E)>0$ such that, for every $m\geq1$,
\[
\lambda^{\otimes m}(D_m(E,\varphi))
     =\lambda(E)^m.
\]
For $m=2k$, every tuple in $D_{2k}(E,\varphi)$
realizes every Boolean pattern, in particular the
alternating pattern with value $0$ on odd coordinates
and value $1$ on even coordinates. Thus
\[
D_{2k}(E,\varphi)\subseteq
\bigcup_{b\in\mathcal U^y}
\bigcap_{i=1}^{k}
\left\{\bar p\in E^{2k}:
F_b(p_{2i-1})=0,\quad F_b(p_{2i})=1\right\}.
\]
Since the functions are $\{0,1\}$-valued, these
conditions are equivalent to
$F_b(p_{2i-1})\leq 1/3$ and
$F_b(p_{2i})\geq 2/3$. Consequently, for every
$k\geq1$, the alternating-pattern set from
Fremlin's definition of stability has full outer
measure relative to $E^{2k}$:
\[
(\lambda^{\otimes2k})^*
D_k(\mathcal A_\varphi,E,1/3,2/3)
=\lambda(E)^{2k}.
\]
Hence $\mathcal A_\varphi$ is not stable in the sense
of \cite[Definition~465B]{Fremlin}. After completing
$\lambda$, Lemma~465L of \cite{Fremlin} provides
measurable $h_0,h_1:X\to[0,1]$ with
\[
\Delta:=\int_X(h_1-h_0)\,d\lambda>0
\]
and the full-outer-measure sets $W_k$ stated in
Fact~\ref{Talagrand}. In particular, there is no
need to replace the variable thresholds $h_0,h_1$
by fixed constants.

The completion introduces no difficulty in the
selection argument. Every function measurable for
the completed regular Borel measure has a Borel
version, equal to it outside a $\lambda$-null set
$N_0$. Replacing $h_0,h_1$ by such versions changes
the definition of $W_k$ only on
\[
\bigcup_{j=1}^{2k}\pi_j^{-1}(N_0),
\]
a set of zero outer product measure. The resulting
$W_k$ therefore still has full outer measure.

\subsection{Finite selection, clopen approximation,
and event pasting}
\label{app:finite-selection}

We give the details behind
Lemma~\ref{lem:finite-threshold}. Fix $k\geq1$
and $0<\eta<1$, and put
$\nu=\lambda^{\otimes2k}$ on $X^{2k}$.
The case $\eta\geq1$ is immediate, since
$\Delta\leq1$ and the desired lower bound
$\Delta-2\eta$ is then negative.
We use Borel versions of $h_0,h_1$ as above.

For $\{0,1\}$-valued functions, the inequalities
\[
F_b(p_{2i-1})\leq h_0(p_{2i-1}),
\qquad
F_b(p_{2i})\geq h_1(p_{2i})
\]
are respectively automatic if
$h_0(p_{2i-1})=1$ or $h_1(p_{2i})=0$;
otherwise they require the literals
$F_b(p_{2i-1})=0$ and $F_b(p_{2i})=1$.
There are only finitely many patterns of
which of the $2k$ conditions are automatic.
They give a finite Borel partition
$(B_s)_{s\in S}$ of $X^{2k}$.
For each $s$, let $\chi_s(\bar p;b)$ be the
conjunction of its non-automatic literals and put
\[
R_{s,b}:=\{\bar p:\chi_s(\bar p;b)\}.
\]
Every $R_{s,b}$ is clopen, because each literal
in $\chi_s$ is a coordinatewise instance of
$\varphi$ with a parameter in $\mathcal U$.
Thus
\[
W_k=\bigcup_{s\in S}
\left(B_s\cap\bigcup_{b\in\mathcal U^y}R_{s,b}\right)
\]
is Borel. Its outer measure is one, by
Fact~\ref{Talagrand}; hence $\nu(W_k)=1$.

For each $s$, the set
$B_s\cap\bigcup_bR_{s,b}$ is Borel.
By inner regularity, choose compact subsets
$K_s$ of these sets with
\[
\sum_{s\in S}\nu(K_s)>1-\frac\eta2.
\]
For each $s$, the open sets $R_{s,b}$ cover $K_s$.
Compactness yields finitely many parameters
$b_{s,1},\ldots,b_{s,m_s}$ covering $K_s$.
Consequently, the Borel set
\[
A:=\bigcup_{s\in S}\bigcup_{j=1}^{m_s}
       (B_s\cap R_{s,b_{s,j}})
\]
satisfies $\nu(A)>1-\eta/2$.

Because $X^{2k}$ is compact and zero-dimensional
and $\nu$ is a regular Borel probability measure,
for each $s$ we may choose a clopen set $D_s$
with
\[
N:=\bigcup_{s\in S}(B_s\triangle D_s),
\qquad \nu(N)<\frac\eta2.
\]
For completeness, clopen approximation follows
by choosing a compact subset $K\subseteq B_s$
and an open superset $U\supseteq B_s$ with
$\nu(U\setminus K)$ small, then covering $K$
by finitely many clopen subsets of $U$.

The finitely many clopen sets
\[
C_{s,j}:=D_s\cap R_{s,b_{s,j}}
\]
can be disjointized, retaining their union.
Write the resulting clopen partition, including
the complement, as
$(H_{s,j})\cup\{H_*\}$.
Every $H_{s,j}$ is a finite Boolean combination
of coordinatewise instances of $\varphi$ with
constant parameters from $\mathcal U$.
Accordingly, evaluating $H_{s,j}$ at
$(f_1,\ldots,f_{2k})$ gives an event
$E_{s,j}$ in the randomization. These events,
together with the complementary event $E_*$,
form a finite measurable partition.

Choose an arbitrary $b_*\in\mathcal U^y$ and
paste the constant random parameters
$b_{s,j}^{\Omega}$ on $E_{s,j}$ and
$b_*^{\Omega}$ on $E_*$. The event-pasting
axioms for $T^R$ yield a random parameter $g$.
No selection from an uncountable family of
parameters is involved: only finitely many
constant random parameters are pasted.
In a concrete randomization, this is a
finite-valued measurable random element.

On $A\setminus N$, there is a unique cell
$B_s$ containing $\bar p$, and the only
possible nonempty $D_t$ at $\bar p$ is $D_s$.
The selected parameter therefore satisfies all
the threshold inequalities at $\bar p$.
Since
\[
\nu(A\setminus N)>1-\eta,
\]
the inequalities can fail only on a set of
$\nu$-measure less than $\eta$.
The probability of each Boolean combination
used in the construction of $g$ is computed by
$\nu$, because the clopen events are determined
by the joint $\varphi$-type distribution of the
$f_j$. Therefore, for every $1\leq i\leq k$,
\[
\begin{aligned}
\mathbb P\llbracket\varphi(f_{2i-1};g)\rrbracket
 &\leq\int_{X^{2k}} h_0(p_{2i-1})\,d\nu+\eta
  =\int_Xh_0\,d\lambda+\eta,\\
\mathbb P\llbracket\varphi(f_{2i};g)\rrbracket
 &\geq\int_{X^{2k}} h_1(p_{2i})\,d\nu-\eta
  =\int_Xh_1\,d\lambda-\eta.
\end{aligned}
\]
Subtracting proves the bound
$\Delta-2\eta$ in Lemma~\ref{lem:finite-threshold}.

\subsection{The coordinate-product law for Morley
sequences of $r_\mu$}
\label{app:morley-law}

Suppose that $\mu$ is definable, let
$p=r_\mu$, and let $(f_i)$ be a Morley
sequence in $p$ over $\mathcal U^\Omega$.
Write $\lambda=\mu_\varphi$ and
$X=S_\varphi(\mathcal U)$.
We verify that the joint $\varphi$-type law of
$f_1,\ldots,f_m$ is
$\lambda^{\otimes m}$ for every finite $m$.

For clopen sets $C_j\subseteq X$, choose
Boolean combinations $\theta_j(x_j)$ of
instances of $\varphi(x_j;b)$, $b\in\mathcal U$,
such that $C_j=[\theta_j]$.
The probability assigned to the clopen
rectangle $C_1\times\cdots\times C_m$ is
\[
\begin{aligned}
&\mathbb P\llbracket
 \bigwedge_{j=1}^{m}\theta_j(f_j)
 \rrbracket\\
&\quad= r_{\mu^{(m)}}
  \left(\mathbb P\llbracket
   \bigwedge_{j=1}^{m}\theta_j(x_j)
   \rrbracket\right)\\
&\quad= \mu^{(m)}
   \left(\bigwedge_{j=1}^{m}\theta_j(x_j)\right)\\
&\quad=\prod_{j=1}^{m}\mu(\theta_j(x))
 =\prod_{j=1}^{m}\lambda(C_j).
\end{aligned}
\]
The first equality involving $r_{\mu^{(m)}}$
uses Fact~\ref{r_mu-fact}(ii).
The factorization of $\mu^{(m)}$ on
separate-coordinate formulas follows
inductively from the definition of the Morley
product: a factor depending on a new coordinate
can be pulled outside the integration over the
previous coordinates.

Clopen rectangles determine the regular Borel
probability measure on the compact Stone space
$X^m$: their finite linear combinations are
dense in $C(X^m)$, so uniqueness follows from
the Riesz representation theorem.
Thus the joint distribution is
$\lambda^{\otimes m}$, as required by
Lemma~\ref{lem:finite-threshold}.
This argument does not assert that the
$f_i$ are pointwise independent random
variables on $\Omega$; it identifies only
their distribution on the indicated space
of classical types.

\subsection{The compactness step in
Theorem~\ref{nondependent-divergent}}
\label{app:compactness-detail}

Work in a monster model of $T^R$ sufficiently
saturated over the countable set of parameters
$\{f_i:i<\omega\}$ and containing
$\mathcal U^\Omega$.
The finite-threshold lemma gives, for every
$k$, a parameter $g_k$ satisfying
\[
\psi(f_{2i};g_k)-\psi(f_{2i-1};g_k)
\geq\frac{3\Delta}{4}
\qquad(1\leq i\leq k).
\]
Therefore every finite subset of the conditions
\[
\psi(f_{2i};y)-\psi(f_{2i-1};y)
\geq\frac{\Delta}{2},\qquad i\geq1,
\]
is satisfiable. Saturation gives a single
parameter $g$ in the same ambient monster
satisfying all conditions.
The choice of the same monster matters: the
assumed convergence of the Morley sequence
is being tested against parameters from
that monster, including $g$.
But then the sequence of real numbers
$\psi(f_i;g)$ cannot converge, because
its successive odd--even differences stay
at least $\Delta/2>0$.
This establishes the contradiction used in
Theorem~\ref{nondependent-divergent}.

\end{document}